\documentclass[regno,12pt]{amsart}
\usepackage{amsmath,amsthm,amssymb,a4,graphics,color}
\usepackage[active]{srcltx}
\usepackage[usenames,dvipsnames]{pstricks}
\usepackage{epsfig}
\usepackage{pst-grad}
\usepackage{pst-plot} 
\usepackage{tikz}
\usepackage{hyperref}
\usepackage{dsfont}
\usepackage{float}

\newfam\Bbbfam 
\font\tenBbb=msbm10 
\font\sevenBbb=msbm7 
\font\fiveBbb=msbm5 
\textfont\Bbbfam=\tenBbb 
\scriptfont\Bbbfam=\sevenBbb 
\scriptscriptfont\Bbbfam=\fiveBbb

\newtheorem{theorem}{Theorem}[section] 
\newtheorem{lemma}[theorem]{Lemma}

\theoremstyle{definition}

\def\1{{\mathchoice {1\mskip-4mu\mathrm l}      
{1\mskip-4mu\mathrm l} 
{1\mskip-4.5mu\mathrm l} {1\mskip-5mu\mathrm l}}}

\renewcommand{\qed}{\hfill\ensuremath{\square}}

\renewcommand{\d}{{\rm d}}

\newcommand{\e}   {{\operatorname e }}

\numberwithin{equation}{section}

\begin{document}
\title{Dormancy in Random Environment: Symmetric Exclusion}
\author[Helia Shafigh]{}
\maketitle
\thispagestyle{empty}
\vspace{-0.5cm}

\centerline{\sc 
Helia Shafigh\footnote{WIAS Berlin, Mohrenstra{\ss}e 39, 10117 Berlin, Germany, {\tt shafigh@wias-berlin.de}}}
\renewcommand{\thefootnote}{}
\vspace{0.5cm}
\centerline{\textit{WIAS Berlin}}

\bigskip


\begin{abstract}
In this paper, we study a spatial model for dormancy in random environment via a two-type branching random walk in continuous-time, where individuals can switch between dormant and active states through spontaneous switching independent of the random environment. However, the branching mechanism is governed by a random environment which dictates the branching rates, namely the simple symmetric exclusion process. We will interpret the presence of the exclusion particles either as \emph{catalysts}, accelerating the branching mechanism, or as \emph{traps}, aiming to kill the individuals. The difference between active and dormant individuals is defined in such a way that dormant individuals are protected from being trapped, but do not participate in migration or branching.

We quantify the influence of dormancy on the growth resp.\ survival of the population by identifying the large-time asymptotics of the expected population size. The starting point for our mathematical considerations and proofs is the parabolic Anderson model via the Feynman-Kac formula. In particular, the quantitative investigation of the role of dormancy is done by extending the Parabolic Anderson model to a two-type random walk.  
 \end{abstract}



\section{Introduction and main results}
\subsection{Biological Motivation} Dormancy is an evolutionary trait that has developed independently across various life forms and is particularly common in microbial communities. To give a definition, we follow \cite{blath} and refer to dormancy as the \emph{ability of individuals to enter a reversible state of minimal metabolic activity}. The collection of all dormant individuals within a population is also often called a \emph{seed-bank}. Maintaining a seed-bank leads to a decline in the reproduction rate, but it also reduces the need for resources, making dormancy a viable strategy during unfavourable periods. Initially studied in plants as a survival strategy (cf.\,\cite{cohen}), dormancy is now recognized as a prevalent trait in microbial communities with significant evolutionary, ecological, and pathogenic implications, serving as an efficient strategy to survive challenging environmental conditions, competitive pressure, or antibiotic treatment. However, it is at the same time a costly trait whose maintenance requires energy and
a sophisticated mechanisms for switching between active and dormant states. Moreover, the increased survival rate of dormant individuals must be weighed against their low reproductive activity. Despite its costs, dormancy still seems to provide advantages in variable environments. For a recent overview on biological dormancy and seed-banks we refer to \cite{dormancy}.

The existing stochastic models for dormancy can be roughly categorized into two approaches: population genetic models and population dynamic models. While the first approach assumes a constant population size and focusses on the genealogical implications of seed-banks, the latter is typically concerned with individual-based modelling through the theory of branching processes. Following a brief example in the book \cite{bookhaccou}, a two-type branching process (without migration) in a fluctuating random environment has been introduced in \cite{blath}, which served as a motivation for this paper. In \cite{blath}, the authors consider three different switching strategies between the two types (dormant and active), namely the \emph{stochastic (or: spontaneous; simultaneous) switching}, \emph{responsive switching} and \emph{anticipatory switching}. In the latter two strategies, individuals adapt to the fluctuating environment by selecting their state (dormant or active) based on environmental conditions via e.g.\,an increased reproduction activity during beneficial phases and a more extensive seed-bank during unfavourable ones in the responsive strategy resp.\,vice versa in the anticipatory strategy. In contrast, the stochastic switching strategy, which remains unaffected by environmental changes, proves especially advantageous during catastrophic events, as it, with high probability, ensures the existence of dormant individuals, which may contribute to the survival of the whole population, when a severely adverse environment might eradicate all active ones. As an example, it is estimated that more than $80\%$ of soil bacteria are
metabolically inactive at any given time, forming extensive seed-banks of dormant individuals independent of the current conditions (cf. \cite{jl11} and \cite{ls18}). This makes the understanding of the stochastic switching strategy an interesting and important task. 

\subsection{Modelling Approach and Goals}  
The aim of this paper is to investigate the stochastic switching strategy in order to quantitatively compare the long-term behaviour of populations with and without this dormancy mechanism, when the underlying environment is random and given by a \emph{simple symmetric exclusion process}, which will be rigorously defined later. 

Inspired by the Galton-Watson process with dormancy introduced in \cite{blath}, a spatial model for dormancy in random environment has been recently introduced in \cite{shafigh}, in which the effect of dormancy on the population growth resp. survival of a population on $\mathbb{Z}^d$ is quantified by identifying the large-time asymptotics of the expected population size. The random environment, which drives the population dynamics, was modelled through three different particle systems: \emph{Bernoulli field of immobile particles}, \emph{one moving particle}, and a \emph{Poisson field of moving particles}. Thus, extending this framework to another classical particle system, such as the \emph{simple symmetric exclusion process}, seems to be a natural step. Moreover, as will be discussed later, the simple exclusion exhibits an interesting \emph{clumping properties} in the lower dimensions $d\leq 2$. Specifically, a finite region occupied by particles takes longer to empty, and conversely, a vacant region takes longer to become occupied. This slower dynamic in lower dimensions makes the exclusion process a particularly suitable choice for modelling the environment, as real-world environments (e.g.\,seasonal changes) often evolve gradually rather than abruptly.

To the best of our knowledge, other spatial models for dormancy in random environment in the setting of \emph{population size} models are still missing. 
 
\subsection{Description of the Model}
In our model, the population lives on $\mathbb{Z}^d$ and consists of two different types $i\in\{0,1\}$ of particles, where we refer to $0$ as \emph{dormant} and to $1$ as \emph{active}. Let $\eta(x,i,t)$ be the number of particles in spatial point $x\in\mathbb{Z}^d$ and state $i$ at time $t\geq 0$, which shall evolve in time according to the following rules:
\begin{itemize}
\item at time $t=0$, there is only one active particle in $0\in\mathbb{Z}^d$ and all other sites are vacant;
\item all particles act independently of each other;
\item active particles become dormant at rate $s_1\geq 0$ and dormant particles become active at rate $s_0\geq 0$;
\item active particles split into two at rate $\xi^+(x,t)\geq 0$ and die at rate $\xi^-(x,t)\geq 0$, depending on their spatial location $x$ and on time $t$, where both $\xi^+$ and $\xi^-$ are random fields;
\item active particles jump to one of the neighbour sites with equal rate $\kappa\geq 0$;
\item dormant particles do not participate in branching, dying or migration.
\begin{figure}[H]
\centering
\includegraphics[scale=0.8]{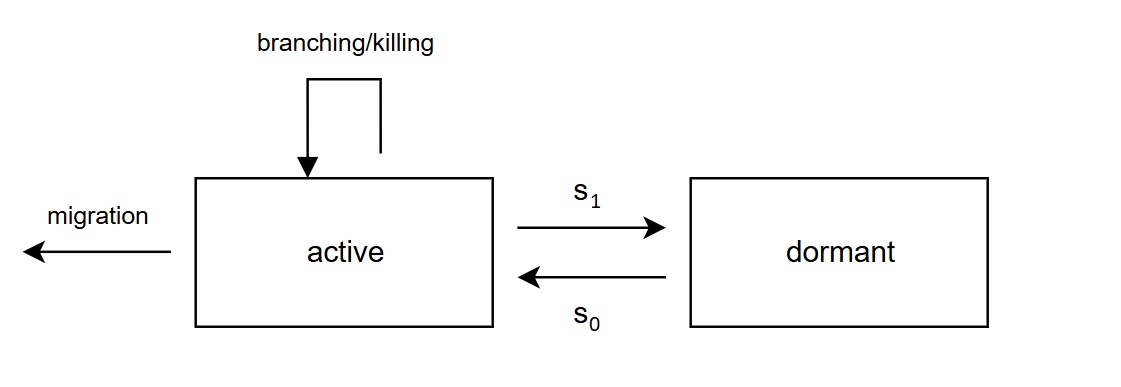}\caption{The evolution in every single point. Active individuals are subject to migration, branching and switching to dormant. Dormant individuals can only get active.}
\end{figure}
\end{itemize}
By assumption, the initial condition is given by $\eta(x,i,0)=\delta_{(0,1)}(x,i)$. Let us define $\eta(t):=\left\{\eta(x,i,t)\mid (x,i)\in\mathbb{Z}^d\right\}$ as configurations on $\mathbb{N}^{\mathbb{Z}^d\times\{0,1\}}$, representing the number of particles in each point $x\in\mathbb{Z}^d$ and state $i\in\{0,1\}$ at time $t$. Then $\eta=(\eta(t))_{t\geq 0}$ is a Markov process on $\mathbb{N}^{\mathbb{Z}^d\times\{0,1\}}$. However, as we will see later, we will use other methods throughout the paper to describe our population, such that a further formalization of $\eta$ shall not be necessary. In the following, we abbreviate $\xi(x,t):=\xi^+(x,t)-\xi^-(x,t)$ for the \emph{balance} between branching and dying and refer to $\xi$ as the underlying \emph{random environment}. In the following, if we fix a realization of $\xi$, then we will denote by
\begin{align}\label{udef}
u(x,i,t):=u_{\xi}(x,i,t):=\mathbb{E}[\eta(x,i,t)\mid \xi]
\end{align}
the expected number of particles in $x\in\mathbb{Z}^d$ and state $i\in\{0,1\}$ at time $t$ with initial condition 
\begin{align*}
u(x,i,0)=\delta_{(0,1)}(x,i),
\end{align*}
where the expectation is only taken over switching, branching and dying (i.\,e.\, over the evolution of $\eta$ for fixed $\xi$) and not over the random environment $\xi$. If we average over $\xi$ as well, what we will denote in the following by $\left<\cdot\right>$, then we refer to
\begin{align*}
\left<u(x,i,t)\right>
\end{align*}
as the \emph{annealed} number of particles in $x\in\mathbb{Z}^d$ and in state $i\in\{0,1\}$ at time $t$. 

\subsection{Simple Exclusion as Random Environment}
As the underlying random environment, we are going to consider the \emph{simple symmetric exclusion process} $\xi=(\xi(x,t))_{x\in\mathbb{Z}^d,t\geq 0}$, a Markov process on $\{0,1\}^{\mathbb{Z}^d}$ with generator (cf\,\cite{liggett})
\begin{align}\label{SE}
L_{\text{SE}}f(\eta)=\rho\sum_{\substack{x,y\in\mathbb{Z}^d\\x\sim y}}\eta(x)(1-\eta(y))(f(\eta^{x,y})-f(\eta))=\rho\sum_{\substack{x,y\in\mathbb{Z}^d\\x\sim y}}(f(\eta^{x,y})-f(\eta)),
\end{align}
for suitable test functions $f:\{0,1\}^{\mathbb{Z}^d}\to\mathbb{R}$, where
\begin{align*}
\eta^{x,y}(z)=\left\{\begin{array}{ll}
\eta(z), &z\neq x,y,\\\eta(y), &z=x,\\\eta(x), &z=y.
\end{array}\right.
\end{align*}
In words, if there is a particle located at site $x\in\mathbb{Z}^d$, it attempts to jump to a neighbouring site $y\in\mathbb{Z}^d$ with rate $\rho$, where the attempt is only successful if the site $y$ is vacant. Reformulating this mechanism, and as we can see from the right hand-side of \eqref{SE}, we can say that two neighbouring sites $x$ and $y$ change states (vacant or occupied by a particles) at rate $\rho$. 
In the following, we will always assume that $\xi$ starts under a Bernoulli product measure $\nu_p:=\nu$ with density $p\in(0,1)$, which is known to be an equilibrium measure for the exclusion dynamics. The graphical representation shows that the evolution is invariant under time reversal and the equilibrium measure $\nu$ is reversible.

Note, that $\xi$ is a non-negative number, which results always in an positive balance between branching and killing. To allow for negative rates as well, we multiply $\xi$ with some factor $\gamma\in[-\infty,\infty)$ and will consider $\gamma\xi$ as the underlying random environment. Thus, the exclusion process can be either interpreted as a field of \emph{traps}, which corresponds to $\gamma<0$, or \emph{catalysts}, if $\gamma>0$. In the first case, active individuals will die with rate $|\gamma|$ if they encounter one of the traps, whereas they branch into two with rate $\gamma$ in the presence of catalysts in the latter case.   

\subsection{Results}
Recall the number of particles $u(x,i,t)$ in point $x\in\mathbb{Z}^d$ and state $i\in\{0,1\}$ at time $t$, as defined in \eqref{udef}. The quantity we are interested in at most in the current paper is the \emph{annealed expected number of all particles}
\begin{align}\label{grossu}
\left<U(t)\right>:=\sum_{x\in\mathbb{Z}^d}\sum_{i\in\{0,1\}}\left<u(x,i,t)\right>,
\end{align}
which turns into the \emph{annealed survival probability} up to time $t$ for $\gamma<0$. Our results concern the large-time asymptotics of $\left<U(t)\right>$ in case of both positive and negative $\gamma$. However, our first result is related to the survival probability in case of no dormancy, when the random walk $X$ is active the whole time, i.e.\,$s_1=0$:
\begin{theorem}[Survival probability without dormancy]
Let $s_1=0$. Then, for all $\gamma\in(-\infty, 0)$, the annealed survival probability $\left<U(t)\right>$ converges to zero as $t\to\infty$ and satisfies the asymptotics
\begin{align}\label{survwithout}
\log\left<U(t)\right>=\left\{\begin{array}{ll}\displaystyle-4K_{p,1}\sqrt{\frac{\rho}{\pi}}\sqrt{t}(1+o(1)), &d=1,\\[13pt]\displaystyle-4K_{p,2}\rho\pi\frac{t}{\log(t)}(1+o(1)), &d=2,\\[13pt]\displaystyle-\lambda_{d,\gamma,\rho,p,\kappa}t(1+o(1)), &d\geq 3,\end{array}\right.
\end{align}
as $t\to\infty$, for some constants $K_{p,1},K_{p,2}\in[p, \log(1-p)]$ and 
\begin{align*}
\lambda_{d,\gamma,\rho, p,\kappa}\geq \frac{p\rho}{\frac{\rho}{|\gamma|}+G_d(0)},
\end{align*}
where $G_d(0)$ denotes the Green's function of a simple symmetric random walk with generator $\Delta$. 
\end{theorem}
In the next theorem the survival probability in case of the stochastic dormancy mechanism is established:
\begin{theorem}[Survival probability with dormancy]
For all $s_0,s_1>0$ and all $\gamma\in(-\infty,0)$ we have
\begin{align*}
\log\left<U(t)\right>=\left\{\begin{array}{ll}\displaystyle-4K_{p,1}\sqrt{\frac{s_0\rho}{(s_0+s_1)\pi}}\sqrt{t}(1+o(1)), &d=1,\\[13pt]\displaystyle-4K_{p,2}\frac{s_0\rho\pi}{s_0+s_1}\frac{t}{\log(t)}(1+o(1)), &d=2,\\[13pt]\displaystyle-\lambda_{d,\gamma,\rho,p,s_0,s_1}t(1+o(1)), &d\geq 3,\end{array}\right.
\end{align*}
as $t\to\infty$, with a constant $\lambda_{d,\gamma, \rho,p,s_0,s_1}$ satisfying
\begin{align*}
\lambda_{d,\gamma, \rho,p,s_0,s_1}\geq \frac{p\rho}{\frac{\rho}{|\gamma|}+G_d(0)}-C,
\end{align*}
for some $C>0$. 
\end{theorem}
In the catalytic case $\gamma>0$, we first state a variational formula for the exponential growth rate of the population size:
\begin{theorem}[Variational formula for the population size]
For all $\gamma>0$ and all $d\geq 1$ the annealed number of particles $\left<U(t)\right>$ grows exponentially as $t\to\infty$ with rate
\begin{align}\label{varfor}
\lim_{t\to\infty}\frac{1}{t}\log\left<U(t)\right>=\sup_{\substack{f\in\ell^2\left(\{0,1\}^{\mathbb{Z}^d}\times\mathbb{Z}^d\times\{0,1\}\right), \\\|f\|_2=1}}(A_1(f)-A_2(f)-A_3(f)-A_4(f))+\sqrt{s_0s_1},
\end{align}
where
\begin{align*}
A_1(f)&:=\int_{\{0,1\}^{\mathbb{Z}^d}}\nu(\d\eta)\sum_{z\in\mathbb{Z}^d}(i\gamma\eta(z)-s_i)f(\eta,z,1)^2
\\A_2(f)&:=\int_{\{0,1\}^{\mathbb{Z}^d}}\nu(\d\eta)\sum_{z\in\mathbb{Z}^d}\sum_{i\in\{0,1\}}\frac{1}{2}\sum_{x,y\in\mathbb{Z}^d, x\sim y}\rho(f(\eta^{x,y},z,i)-f(\eta,z,i))^2,
\\A_3(f)&:=\int_{\{0,1\}^{\mathbb{Z}^d}}\nu(\d\eta)\sum_{z\in\mathbb{Z}^d}\frac{1}{2}\sum_{y\in\mathbb{Z}^d,y\sim z}\kappa(f(\eta,y,1)-f(\eta,z,1))^2,
\\A_4(f)&:=\int_{\{0,1\}^{\mathbb{Z}^d}}\nu(\d\eta)\sum_{z\in\mathbb{Z}^d}2\sqrt{s_0s_1}(f(\eta,z,1)-f(\eta,z,0))^2.
\end{align*}
\end{theorem}
The next theorem quantifies the explicit growth rate in the lower dimensions $d\in\{1,2\}$ and provides bound for the growth rate in dimensions $d\geq 3$:
\begin{theorem}[Growth rate]
Let $\gamma\in(0,\infty)$ .
\begin{itemize}
\item[(a)] If $d\in\{1,2\}$, then
\begin{align*}
\lim_{t\to\infty}\frac{1}{t}\log\left<U(t)\right>=\gamma-s_1-\frac{(\gamma+s_0-s_1)^2-s_0s_1}{\sqrt{\gamma^2+2\gamma(s_0-s_1)+(s_0+s_1)^2}}.
\end{align*}
\item [(b)] If $d\geq 3$, then
\begin{align*}
\gamma> \lim_{t\to\infty}\frac{1}{t}\log\left<U(t)\right>>\left\{\begin{array}{ll}
\gamma p, &s_1\leq s_0,\\\gamma p -s_1+s_0, &s_1>s_0 \text{ and } \gamma p-s_1+s_0\geq 0,\\0, &s_1>s_0 \text{ and } \gamma p-s_1+s_0<0.\end{array}\right.
\end{align*}
\end{itemize}
\end{theorem}
\subsection{Relation to the Parabolic Anderson Model}
Recall the number of particles $u(x,i,t)$ in point $x\in\mathbb{Z}^d$ and state $i\in\{0,1\}$ at time $t$ as defined in \eqref{udef}. It is already known (cf.\,\cite{baran}) that $u:\mathbb{Z}^d\times\{0,1\}\times[0,\infty)\to\mathbb{R}$ solves the partial differential equation
\medskip
\begin{align}\label{pamswitching}
\left\{\begin{array}{lllr}\frac{\d}{\d t}u(x,i,t)&=&i\kappa\Delta u(x,i,t) + Q u(x,i,t)+i\gamma\xi(x,t)u(x,i,t), &t>0, \\[12pt]u(x,i,0)&=&\delta_{(0,1)}(x,i),
\end{array}\right.
\end{align}
where
\begin{align*}
Q u(x,i,t):= s_i(u(x,1-i,t)-u(x,i,t))
\end{align*}
and
\begin{align*}
\Delta u(x,i,t):=\sum_{y\in\mathbb{Z}^d, x\sim y}[u(y,i,t)-u(x,i,t)].
\end{align*}
We call \eqref{pamswitching} the \emph{parabolic Anderson model with switching}. The parabolic Anderson model \emph{without} switching, i.\,e.\ with only one (active) type, has been studied intensely during the past years and a comprehensive overview of results can be found in \cite{PAM}. One of the most powerful tools and often the starting point of the analysis of the PAM is the \emph{Feynman-Kac formula}, which asserts that the time evolution of all particles up to a deterministic time can be expressed as an expectation over single particle moving around according to the same migration kernel and with a \emph{varying mass}, which corresponds to the underlying population size. Let us formulate the Feynman-Kac formula in the setting of our dormancy strategy. To this end, let $\alpha=(\alpha(t))_{t\geq 0}$ be a continuous-time Markov process with state space $\{0,1\}$ and generator
\begin{align}\label{Q}
Qf(i):=s_i(f(1-i)-f(i))
\end{align}
for $f:\{0,1\}\to\mathbb{R}$. Conditioned on the evolution of $\alpha$, we define a continuous-time random walk $X=(X(t))_{t\geq 0}$ on $\mathbb{Z}^d$ which is supposed to stay still at a time $t$, if $\alpha(t)=0$, or perform a simple symmetric walk with jump rate $2d\kappa$, if $\alpha(t)=1$. In other words, the joint process $(X,\alpha)$ is the Markov process with the generator
\begin{align}\label{Lxa}
\mathcal{L}f(x,i):=i\kappa\sum_{y\sim x}(f(y,i)-f(x,i))+s_i(f(x,1-i)-f(x,i))
\end{align} 
for $x\in\mathbb{Z}^d$, $i,j\in\{0,1\}$ and a test function $f:\mathbb{Z}^d\times \{0,1\}\to\mathbb{R}$. Note, that the random walk $X$ itself is not Markovian due to the dependence on $\alpha$. Then, we call $(X,\alpha)$ a \emph{regime-switching random walk} (cf.\,\cite{switching} for the continuous-space version) and interpret $X$ as a particle which is \emph{active} at time $t$, if $\alpha(t)=1$, and \emph{dormant} otherwise. Then, given a fixed realization of $\xi$, the formal solution of \eqref{pamswitching} is given by the \emph{Feynman-Kac formula}
\begin{align}\label{fkformula}
u(x,i,t)=\mathbb{E}_{(x,i)}^{(X,\alpha)}\left[\exp\left(\int_0^t\gamma\alpha(s)\xi(X(s),t-s)\,\d s\right)\delta_{(0,1)}(X(t),\alpha(t))\right],
\end{align}
where $\mathbb{E}_{(x,i)}^{(X,\alpha)}$ denotes the expectation over the joint process $(X,\alpha)$ starting in $(x,i)$ (cf.\,\cite{baran}). Thus, the study of our two-type branching process can be reduced to the analysis of only one particle with the same migration, branching and switching rates.

\subsection{Related Results}
The parabolic Anderson model without switching has been a topic of great interest during the past years and has been studied for several different random environments built out of particles. For a recent overview of results related to the Parabolic Anderson model we refer to \cite{PAM}. In the catalytic case $\gamma>0$, the simple symmetric exclusion process as a random environment has been investigated in \cite{exclusion}, in which the authors prove exponential growth of the annealed population size $\left<U(t)\right>$ as $t\to\infty$. More precisely, in \cite[Proposition 2.2.2]{exclusion} it has been proven that
\begin{align*}
\lim_{t\to\infty}\frac{1}{t}\log\left<U(t)\right>=\sup_{\substack{f\in\ell^2\left(\{0,1\}^{\mathbb{Z}^d}\times\mathbb{Z}^d\right), \\\|f\|_2=1}}(A_1(f)-A_2(f)-A_3(f))
\end{align*}
with
\begin{align*}
A_1(f)&:=\int_{\{0,1\}^{\mathbb{Z}^d}}\nu(\d\eta)\sum_{x\in\mathbb{Z}^d}\eta(x)f(\eta,x)^2,
\\A_2(f)&:=\int_{\{0,1\}^{\mathbb{Z}^d}}\nu(\d\eta)\sum_{z\in\mathbb{Z}^d}\frac{1}{2}\sum_{x,y\in\mathbb{Z}^d, x\sim y}\rho(f(\eta^{x,y},z)-f(\eta,z))^2,
\\A_3(f)&:=\int_{\{0,1\}^{\mathbb{Z}^d}}\nu(\d\eta)\sum_{z\in\mathbb{Z}^d}\frac{1}{2}\sum_{y\in\mathbb{Z}^d,y\sim z}\kappa(f(\eta,y)-f(\eta,z))^2,
\end{align*}
for all dimensions $d\geq 1$. Moreover, \cite[Theorem 1.3.2(a)]{exclusion} asserts that, under our assumptions as defined in section 1.4, we have that
\begin{align*}
\lim_{t\to\infty}\frac{1}{t}\log\left<U(t)\right>\left\{\begin{array}{ll}=\gamma, &d\in\{1,2\},\\\in(p\gamma, \gamma), &d\geq 3.
\end{array}\right.
\end{align*}
In the trapping case $\gamma<0$ and in the higher dimensions $d\geq 3$, the upper bound on the survival probability of a random walk in a dynamic random trap model has been studied in \cite{upper} within a more general framework. More precisely, it was shown that
\begin{align*}
\limsup_{t\to\infty}\frac{1}{t}\log\left<U(t)\right><0
\end{align*}
under certain assumptions on the environment, which are satisfied for the simple symmetric exclusion. To the best of our knowledge, this result is currently the only one addressing the large-time asymptotics of the survival probability in the context of the exclusion process as a random environment. Notably, precise asymptotics in the lower dimensions $\d\leq 2$ remain unexplored, such that our Theorem 1.1 appears to fill a gap in the study of the Parabolic Anderson model (without switching) driven by the exclusion process as random environment. 

Recently, the Parabolic Anderson model with the stochastic dormancy strategy as defined in section 1.3  has been studied in \cite{shafigh} for some specific choices of the random environment $\xi$. More precisely, in both cases $\gamma>0$ and $\gamma<0$, the annealed number of particles resp.\,the asymptotics of annealed survival probability $\left<U(t)\right>$ has been quantified for the cases when $\xi$ is given by 1) a Bernoulli field of immobile particles, 2) one moving particles, and 3) a Poisson field of independently moving particles. In the latter case, it has been shown that, if $\gamma\in[-\infty,0)$, the annealed survival probability $\left<U(t)\right>$ converges exponentially fast to $0$ as $t\to\infty$ in all dimensions $d\geq 1$ and obeys the asymptotics
\medskip
\begin{align}\label{asy3-}
\log\left<U(t)\right>=\left\{\begin{array}{ll}\displaystyle-4p\sqrt{\frac{\rho s_0}{(s_0+s_1)\pi}}\sqrt{t}(1+o(1)), &d=1,\\[13pt]\displaystyle-4p\frac{\rho\pi s_0}{s_0+s_1}\frac{t}{\log\left(t\right)}(1+o(1)), &d=2,\\[13pt]\displaystyle-\mu_{d,\gamma, \rho,p,s_0,s_1} t(1+o(1)), &d\geq 3,\end{array}\right.
\end{align}
as $t\to\infty$, for some constant $\mu_{d,\gamma, \rho,p,s_0,s_1}>0$ depending on all the parameters. Interestingly, the decays rates, at least in dimensions $d\in\{1,2\}$, equal our decay rates in case of the exclusion process. However, for $\gamma>0$, the large-time behaviour of the population size $\left<U(t)\right>$ does not appear to closely align with the large-time asymptotics observed in the simple exclusion process. More precisely, $\left<U(t)\right>$ has been shown in \cite{shafigh} to grow double-exponentially fast with limiting rate given by
\begin{align}\label{asy3+}
\lim_{t\to\infty}\frac{1}{t}\log\log\left<U(t)\right>=\sup_{f\in\ell^2(\mathbb{Z}^d),\|f\|_2=1}\left(\gamma f(0)^2-\frac{1}{2}\sum_{x,y\in\mathbb{Z}^d, x\sim y}\rho(f(x)-f(y))^2\right),
\end{align}
which coincides with the growth rate without dormancy, as studied in \cite{drewitz}. In contrast, \cite[Theorem 1.1(b)]{shafigh} asserts that the annealed population size $\left<U(t)\right>$ in case of a Bernoulli field of immobile particles grows exponentially fast in all dimensions $d\geq 1$ with rate given by
\begin{align*}
\lim_{t\to\infty}\frac{1}{t}\log\left<U(t)\right>=\gamma-s_1-\frac{(\gamma+s_0-s_1)^2-s_0s_1}{\sqrt{\gamma^2+2\gamma(s_0-s_1)+(s_0+s_1)^2}},
\end{align*}
which coincides with our growth rate in the case of the exclusion process in the lower dimensions $d\in\{1,2\}$. This comparison suggests that the exclusion process in lower dimension tends to stay close to the initial Bernoulli distribution to maximize the growth rate. As we will demonstrate in the proofs, this behaviour is a consequence of the fact that, in lower dimensions, remaining stationary incurs a lower cost for the exclusion process relative to the corresponding growth rate.
\section{Preparatory facts}
In the following, we will consistently use the notation $\mathbb{P}^{Y}_\mu$ and $\mathbb{E}^{Y}_\mu$ to denote the distribution and expectation, respectively, with respect to any random sequence $Y:=(Y(t))_{t\geq 0}$ with initial measure $\mu$. For real-valued sequences with a delta initial distribution $\mu=\delta_y$, we will write $\mathbb{P}^{Y}_y$ and $\mathbb{E}^{Y}_y$, respectively, and this convention will be used throughout the paper without further clarification. 
\subsection{Total population size}
As discussed in the introduction, our proofs and considerations regarding the total population size are founded on the Feynman-Kac formula \eqref{fkformula}, which serves as the cornerstone for the subsequent steps throughout the remainder of this paper. Recall, that the symmetric exclusion is reversible in time under $\nu$, in the sense that $(\xi(\cdot,t))_{0\leq t\leq T}$ is equally distributed to $(\xi(\cdot,T-t))_{0\leq t\leq T}$, for all $T>0$. Hence, in the same manner as in \cite[Section 2.1]{shafigh}, we observe that the total population size can be expressed as 
\begin{align*}
\left<U(t)\right>=\mathbb{E}_\nu^\xi\mathbb{E}_{(0,1)}^{(X,\alpha)}\left[\exp\left(\int_0^t\gamma\alpha(s)\xi(X(s),s)\,\d s\right)\right].
\end{align*}
\subsection{Symmetric Exclusion: Large deviations and the environmental process}
In this section we summarize a few known properties of the symmetric exclusion process, which will be used in the proofs of our results. We start with some notations. For a test function $f\in\ell^2(\{0,1\}^{\mathbb{Z}^d})$, where the integral is taken with respect to the invariant measure $\nu$, write
\begin{align*}
\mathcal{D}(f):=\left<-L_{\text{SE}}f,f\right>&=-\int_{\{0,1\}^{\mathbb{Z}^d}}\nu(\d\eta)\rho\sum_{x\in\mathbb{Z}^d}\sum_{\substack{y\in\mathbb{Z}^d\\x\sim y}}(f(\eta^{x,y})-f(\eta))f(\eta)
\\&=\frac{1}{2}\rho\int_{\{0,1\}^{\mathbb{Z}^d}}\nu(\d\eta)\sum_{\substack{x,y\in\mathbb{Z}^d\\x\sim y}}(f(\eta^{x,y})-f(\eta))^2
\end{align*}
for the Dirichlet form associated with $L_{\text{SE}}$ and define
\begin{align}\label{Ixi}
I_\xi(x):=\inf_{\substack{\mu\in\mathcal{M}_1\left(\{0,1\}^{\mathbb{Z}^d}\right)\\\int_{\{0,1\}^{\mathbb{Z}^d}}\mu(\d\eta)\eta(0)=x}}\lim_{\varepsilon\to 0}\inf_{\substack{\varphi\in B(\mu,\varepsilon)\\\varphi\ll\nu}}\mathcal{D}\left(\sqrt{\frac{\d\varphi}{\d\nu}}\right),
\end{align}
where we denote by $\mathcal{M}_1\left(\{0,1\}^{\mathbb{Z}^d}\right)$ the set of probability measures on $\{0,1\}^{\mathbb{Z}^d}$ and by $B(\mu,\varepsilon)$ the open ball of radius $\varepsilon$ around $\mu$. Further, let
\begin{align*}
T_t:=\int_0^t\xi(0,s)\d s
\end{align*} 
be the \emph{local time} of the exclusion process in the origin up to time $t$. Then, it is already known from \cite{landim} that the normalized local times $(\small\frac{1}{t}T_t)$ satisfy a large deviation principle on scale $t$ and with rate function $I_\xi$ in dimensions $d\geq 3$, which is non-trivial in these dimensions as seen from the Dirichlet form, as well as convex and lower semi-continuous with a unique zero at $x=\rho$. However, in the lower dimensions $d\in\{1,2\}$ the rate function \eqref{Ixi} equals zero due to the recurrence of the random walk, such that a different scaling than $t$ would be needed. Regarding this, in \cite{landim} and \cite{arratia} it has been shown that $(\small\frac{1}{t}\int_0^t\xi(0,s)\,\d s)_{t>0}$ satisfies a large deviation principle on $[0,1]$, as $t\to\infty$, with speed
\begin{align}\label{scales}
a_t:=\left\{\begin{array}{ll}\sqrt{t}, &d=1,\\\frac{t}{\log t}, &d=2,
\end{array}\right.
\end{align}
and a non-trivial rate function, which we shall not need in the following. In \cite{exclusion}, similar methods have been used to derive a lower bound with the same scales as \eqref{scales} on the occupation probability of a finite box. More precisely, \cite[Lemma 3.1.1]{exclusion} asserts that in dimensions $d\in\{1,2\}$ and for any finite box $Q\subseteq \mathbb{Z}^d$, 
\begin{align*}
\mathbb{P}_\nu^\xi(\xi(x,s)=i \text{ for all } x\in Q \text{ for all } s\leq t)\geq \e^{-C_d|Q|a_t}
\end{align*}
for some constants $C_1,C_2>0$. In other words, the exclusion process exhibits a \emph{clumping property} in lower dimensions, meaning that it takes significantly longer for a finite box $Q$ filled with particles to empty, or for a vacant box $Q$ to become filled with particles, relative to the time scale $t$.

An important and useful concept in the study of random walks on top of the exclusion process is the \emph{environmental process}, also known as the \emph{environment as seen from the walker}. This process captures how the environment evolves at the locations visited by the random walk. In our case, as the random walk depends on $\alpha$, the environmental process is defined as the Markov process $(\zeta, \alpha)$ with generator
\begin{align*}
L_{\text{EP}}f(\eta, i)=L_{\text{SE}}f(\cdot,i)(\eta)+Qf(\eta,\cdot)(i)+
\sum_{y\sim 0}i\kappa(f(\tau_y\eta,i)-f(\eta,i)),
\end{align*}
where $\tau_y$ denotes the shift-operator in $y$, i.\,e.\,,
\begin{align}\label{shift}
\tau_y\eta(z)=\eta(z+y).
\end{align}
In other words, 
\begin{align}\label{zeta}
\zeta(x,t)=\tau_{X(t)}\xi(x,t)=\xi(x+X(t),t),
\end{align}
such that the environmental process changes whenever the random walk moves or the configuration of the exclusion process changes. It is worth mentioning that the Bernoulli product measure $\nu$ remains an invariant measure for the environmental process (cf.\,\cite{liggett}).
\subsection{Switching mechanism: Large deviations and change of measure} 
Let us recall a large deviation principle which we will frequently use throughout the paper. In the following, we define
\begin{align*}
L_t(i):=\int_0^t\alpha(s)\d s, \qquad i\in\{0,1\},
\end{align*}
as the \emph{local times} of the Markov chain $\alpha$ in state $i$ up to time $t>0$. It was shown in \cite[Corollary 2.2]{shafigh} that the normalized local times $(\small\frac{1}{t}L_t(1))_{t>0}$ in the \emph{active} state $1$ satisfy a large deviation principle on $[0,1]$ with rate function $I:[0,1]\to\mathbb{R}$ given by
\begin{align}\label{I}
I(a)=-2\sqrt{s_0s_1a(1-a)}+(s_1-s_0)a+s_0,
\end{align}
where we recall that $s_0,s_1\geq 0$ are the switching rates of $\alpha$. This strictly convex rate function has a unique zero in $\small\frac{s_0}{s_0+s_1}$ and is valid for every choise of $s_0,s_1\geq 0$. In particular, when $s_0=s_1$, the rate function simplifies to $I(a)=s(\sqrt{a}-\sqrt{1-a})^2$ which corresponds to the well-known large deviation rate function for symmetric transition rates (cf.\,\cite[Theorem 3.6.1 and Remark 3.6.4]{koenigldp}).

One consequence of the large deviation result for $(\small\frac{1}{t}L_t(1))_{t>0}$ on scale $t$ is the following: for any sequence $(f(t))_{t\geq 0}$ of positive real numbers with $\lim_{t\to\infty}f(t)=\infty$ and $\lim_{t\to\infty}\frac{f(t)}{t}=0$, and for any continuous and bounded function $F\colon[0,1]\to\mathbb{R}$,
\begin{align}\label{lemma2.4}
\lim_{t\to\infty}\frac{1}{f(t)}\log\mathbb{E}_{1}^{\alpha}\left[\e^{f(t)F\left(\frac{1}{t}L_t(1)\right)}\right]=F\left(\frac{s_0}{s_0+s_1}\right).
\end{align}
This is the assertion of \cite[Lemma 2.4]{shafigh} and implies that on any smaller scale than $t$, the best the sequence $(\small\frac{1}{t}L_t(1))_{t>0}$ can do in order to maximize the expectation on the left hand-side of \eqref{lemma2.4} is to take its average value $\small\frac{s_0}{s_0+s_1}$.
 
The fact that our transitions rates are allowed to be asymmetric not only changes the large deviation rate function but may also introduce additional challenges. Specifically, one of the proof techniques we will use later to derive the representation \eqref{varfor} relies on the Perron-Frobenius spectral theory for bounded self-adjoint operators. However, the matrix $Q$, defined in \eqref{Q}, is not symmetric, and consequently, any operator involving $Q$ may fail to be self-adjoint. To address this issue, it was shown in \cite[Corollary 2.6]{shafigh} using a result from \cite{girsanovmarkov} that if $\tilde{\alpha}$ is a Markov process on $\{0,1\}$ with symmetric generator
\begin{align}\label{deftildeq}
\tilde{Q}f(i):=\sqrt{s_0s_1}(f(1-i)-f(i))
\end{align}
for $f\colon\{0,1\}\to\mathbb{R}$, then
\begin{align}\label{LEP}
\frac{\d\mathbb{P}^{\alpha}_1}{\d\mathbb{P}^{\tilde{\alpha}}_1}\Big|_{\mathcal{F}_t}=\exp\left(\sqrt{s_0s_1}t-s_0\tilde{L}_t(0)-s_1\tilde{L}_t(1)\right)
\end{align}
where we wrote $\tilde{L}_t(i)=\int_0^t\delta_i(\tilde{\alpha}(s))\,\d s$ for the local times of $\tilde{\alpha}$ in state $i\in\{0,1\}$ up to time $t$. In particular, if $(\xi,\tilde{X},\tilde{\alpha})$ is the Markov process with symmetric generator
\begin{align*}
\tilde{L}f(\eta,x,i):=L_{\text{SE}}f(\cdot,x,i)(\eta)+i\kappa\sum_{y\sim x}(f(\eta,y,i)-f(\eta,x,i))+\sqrt{s_0s_1}(f(\eta,x,1-i)-f(\eta,x,i))
\end{align*}
for test functions $f\colon\{0,1\}^{\mathbb{Z}^d}\times\mathbb{Z}^d\times\{0,1\}\to\mathbb{R}$, then
\begin{align}\label{rndichte}
\frac{\d\mathbb{P}^{(\xi,X,\alpha)}_{\nu,(0,1)}}{\d\tilde{\mathbb{P}}^{(\xi,\tilde{X},\tilde{\alpha})}_{\nu,(0,1)}}\Big|_{\mathcal{F}_t}=\exp\left(\sqrt{s_0s_1}t-s_0\tilde{L}_t(0)-s_1\tilde{L}_t(1)\right),
\end{align}
since the generator of $X$ conditioned on $\alpha$ matches that of $\tilde{X}$ conditioned on $\tilde{\alpha}$, and since $\alpha$ and $\tilde{\alpha}$ are independent of $X$  and $\tilde{X}$, respectively, as well as of $\xi$. It is straightforward to verify that $\tilde{L}$ is indeed self-adjoint, enabling us to apply the Perron-Frobenius theory. The same reasoning applies for the environmental process $(\zeta,\alpha)$: If $(\tilde{\zeta},\tilde{\alpha})$ is the Markov process with symmetric generator
\begin{align*}
\tilde{L}_{\text{EP}}f(\eta,i):=L_{\text{SE}}f(\cdot,i)(\eta)+\tilde{Q}f(\eta,\cdot)(i)+\sum_{y\sim 0}i\kappa(f(\tau_y\eta,i)-f(\eta,i))
\end{align*}
for test functions $f\colon\{0,1\}^{\mathbb{Z}^d}\times\{0,1\}\to\mathbb{R}$, then
\begin{align}\label{rndichte2}
\frac{\d\mathbb{P}^{(\zeta,\alpha)}_{\nu,1}}{\d\tilde{\mathbb{P}}^{(\tilde{\zeta},\tilde{\alpha})}_{\nu,1}}\Big|_{\mathcal{F}_t}=\exp\left(\sqrt{s_0s_1}t-s_0\tilde{L}_t(0)-s_1\tilde{L}_t(1)\right),
\end{align}
which we will utilize later in the proofs. 
\section{Survival Probability without Dormancy}
This section is devoted to the proof of Theorem 1.1, which established the long-time asymptotics of the survival probability without the dormancy mechanism and therefore without the switching Markov chain $\alpha$. The proof is structured into two lemmas, which establish the upper and lower bound, respectively. 
\begin{lemma}[Upper bound without switching]
For all $\gamma\in(-\infty, 0)$,
\begin{align}\label{lemma31}
\log\mathbb{E}_0^X\mathbb{E}^{\xi}_\nu\left[\exp\left(\gamma\int_0^t\xi(X(s),s)\,\d s\right)\right]\leq \left\{\begin{array}{ll}\displaystyle-4p\sqrt{\frac{\rho}{\pi}}\sqrt{t}(1+o(1)), &d=1,\\[13pt]\displaystyle-4p\rho\pi\frac{t}{\log(t)}(1+o(1)), &d=2,\\[13pt]\displaystyle-\frac{\rho p}{\frac{\rho}{|\gamma|}+G_d(0)}t(1+o(1)), &d\geq 3,\end{array}\right.
\end{align}
as $t\to\infty$, where we write $G_d(0)$ for the Green's function of a simple symmetric random walk with jump rate $2d$ in $0$.
\end{lemma}
\begin{proof}
The proof relies on a comparison inequality between the simple exclusion process and a system of independent random walks, as established in \cite{exclusion}. More precisely, let $\hat{\xi}=(\hat{\xi}(x,t))_{x\in\mathbb{Z}^d,t\geq 0}$ represent the configurations generated by a collection of independent simple symmetric random walks with jump rate $2d\kappa$, where $\hat{\xi}(0)$ is distributed according to $nu$. Then, \cite[Proposition 1.2.1]{exclusion} asserts that for all $K:\mathbb{Z}^d\times[0,\infty)\to\mathbb{R}$ such that either $K\geq 0$ oder $K\leq 0$, and all $t\geq 0$ such that $\sum_{z\in\mathbb{Z}^d}\int_0^t|K(z,s)|\,\d s<\infty$ and all $\eta\in\{0,1\}^{\mathbb{Z}^d}$,
\begin{align}\label{comparison}
\mathbb{E}_\nu^{\xi}\left[\exp\left(\sum_{z\in\mathbb{Z}^d}\int_0^tK(z,s)\xi(z,s)\,\d s\right)\right]\leq \mathbb{E}_\nu^{\hat{\xi}}\left[\exp\left(\sum_{z\in\mathbb{Z}^d}\int_0^tK(z,s)\hat{\xi}(z,s)\,\d s\right)\right].
\end{align}
Now, fix a realization of $X$ with $N_t$ jumps up to time $t$ and let $\tau_1,\tau_2,\cdots,\tau_{N_t}$ be the corresponding jump times. Further, let $x_k$, $k=0,\cdots,N)t$ denote the corresponding values on each interval $[\tau_k,\tau_{k+1})$, where we set $x_0=0$ and $\tau_0=0$. Define the function $K:\mathbb{Z}^d\times[0,t]\to\mathbb{R}$ as
\begin{align*}
K(z,s)=\gamma\delta_{x_k}(z)\quad \text{for all } s\in[\tau_k,\tau_{k+1}), k=1,\cdots,N_t.
\end{align*}
Then, $K\leq 0$ and
\begin{align*}
\sum_{z\in\mathbb{Z}^d}\int_0^t|K(z,s)|\,\d s=|\gamma |t<\infty
\end{align*}
for all $t\geq 0$. Thus, by \eqref{comparison},
\begin{align*}
\mathbb{E}_\nu^\xi\left[\exp\left(\gamma\int_0^t\xi(X(s),s)\,\d s\right)\right]&=\mathbb{E}_\nu^\xi\left[\exp\left(\sum_{z\in\mathbb{Z}^d}\int_0^tK(z,s)\xi(z,s)\,\d s\right)\right]\\&\leq \mathbb{E}_\nu^{\hat{\xi}}\left[\exp\left(\sum_{z\in\mathbb{Z}^d}\int_0^tK(z,s)\hat{\xi}(z,s)\,\d s\right)\right]
\\&=\mathbb{E}_\nu^{\hat{\xi}}\left[\exp\left(\gamma\int_0^t\hat{\xi}(X(s),s)\,\d s\right)\right].
\end{align*}
For each configuration $\eta\subseteq\{0,1\}^{\mathbb{Z}^d}$, let $A_{\eta}:=\{x\in\mathbb{Z}^d : \eta(x)=1\}$. Further, let $Y_y$ denote a random walk with jump rate $2d\rho$ starting from $y\in\mathbb{Z}^d$. We now integrate over the Bernoulli system of independent random walks to obtain
\begin{align*}
\mathbb{E}_\nu^{\hat{\xi}}\left[\exp\left(\gamma\int_0^t\hat{\xi}(X(s),s)\,\d s\right)\right]&=\int_{\{0,1\}^{\mathbb{Z}^d}}\mathbb{E}_{\eta}^{\hat{\xi}}\left[\exp\left(\gamma\int_0^t\sum_{y\in A_\eta}\delta_{X(s)}(Y_y(s))\,\d s\right)\right]\nu(\d\eta)
\\&=\int_{\{0,1\}^{\mathbb{Z}^d}}\prod_{y\in A_\eta}v_X(y,t)\nu(\d\eta),
\end{align*}
where we used the independence of the random walks and abbreviated
\begin{align*}
v_X(y,t):=\mathbb{E}_y^{Y_y}\left[\exp\left(\gamma\int_0^t\delta_{X(s)}(Y_y(s))\d s\right)\right].
\end{align*}
Using the fact that $\nu$ is a Bernoulli product measure, we obtain
\begin{align}\label{ineqX}
\int_{\{0,1\}^{\mathbb{Z}^d}}\prod_{y\in A_\eta}v_X(y,t)\nu(\d\eta)=\prod_{y\in\mathbb{Z}^d}p v_X(y,t)\leq \exp\left(p\sum_{y\in\mathbb{Z}^d}(v_X(y,t)-1)\right).
\end{align}
In \cite{drewitz} it was shown that the right hand-side of \eqref{ineqX} is nothing but the survival probability of $X$ among a Poisson system of moving traps, which is maximized for $\kappa=0$, corresponding to an immobile random walk, and follows the same asymptotics as on the right hand-side of \eqref{lemma31}.
\end{proof}
Next, we prove the lower bound in Theorem 1.1:
\begin{lemma}[Lower bound without switching in $d=1,2$]\label{lowerwithout}
For all $\gamma\in(-\infty,0)$,
\begin{align*}
\log\mathbb{E}_0^X\mathbb{E}_\nu^\xi\left[\exp\left(\gamma\int_0^t\xi(X(s),s)\,\d s\right)\right]\geq \left\{\begin{array}{ll}\displaystyle 4\log(1-p)\sqrt{\frac{\rho}{\pi}}\sqrt{t}(1+o(1)), &d=1,\\[13pt]\displaystyle 4\log(1-p)\rho\pi\frac{t}{\log(t)}(1+o(1)), &d=2,\end{array}\right.
\end{align*}
as $t\to\infty$.
\end{lemma}
\begin{proof}
The proof if based on the \emph{graphical representation} of the exclusion process, which we will denote by $\mathcal{G}$ in the following. In $\mathcal{G}$, space is drawn sidewards, time upwards,
and, at rate $\rho$, we place links between each neighbouring sites $x,y\in\mathbb{Z}^d$. The configuration of the exclusion process at time $t$ is then obtained from the one at time $0$ of $\mathcal{G}$ by transporting the local states
along paths moving upwards with time and sidewards along links:
\begin{figure}[H]
\centering
\includegraphics[scale=0.8]{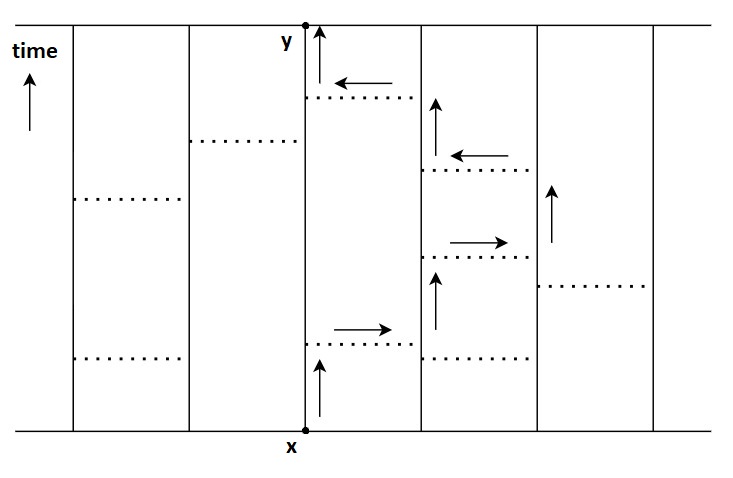}\caption{Graphical representation of the simple exclusion. Arrows represent a path from $x$ to $y$ through the dashed lines, representing links between neighbours evolving with time.}
\end{figure} 
Now, let $Q_t:=[-r_t,r_t]^d$, where we choose $r_t=\sqrt{\frac{t}{\log(t)}}$ for $d=1$ and $r_t=\log(t)$ for $d=2$. The strategy is to impose a condition on the exclusion process to create a vacant area $Q_t$ around zero up to time $t$ and to constrain the random walk $X$ to stay within $Q_t$ throughout the interval $[0,t]$. Define the events
\begin{align*}
A_t:=\left\{\xi(x,s)=0\forall x\in Q_t\forall s\in[0,t]\right\} \quad\text{and}\quad  B_t:=\left\{X(s)\in Q_t\forall s\in[0,t]\right\}.
\end{align*}
Then,
\begin{align*}
\mathbb{E}_0^X\mathbb{E}_\nu^\xi\left[\exp\left(\gamma\int_0^t\xi(X(s),s)\,\d s\right)\right]\geq \mathbb{P}(A_t)\mathbb{P}(B_t).
\end{align*}
In \cite[Lemma 2.1]{drewitz} it was shown that there exist some constant $\beta>0$ such that
\begin{align}\label{asB}
\log\mathbb{P}(B_t)\geq \log(\beta)\frac{t}{r_t^2}=\left\{\begin{array}{ll}\log(\beta)\log(t), &d=1,\\\log(\beta)\frac{t}{\log(t)^2}, &d=2.
\end{array}\right.
\end{align}
In order to determine $\mathbb{P}(A_t)$, we define
\begin{align*}
H_t^{Q_t}:=\left\{x\in\mathbb{Z}^d : \exists \text{ path in } \mathcal{G} \text{ from } (0,x) \text{ to } [0,t]\times Q_t\right\} 
\end{align*}
and $H_t^0$ analogously with $Q_t$ replaced by $\{0\}$. Further, denote by $\mathcal{P}$ and $\mathcal{E}$ the probability and expectation, respectively, with respect to $\mathcal{G}$. According to \cite{arratia},
\begin{align}\label{asA}
\mathbb{P}(A_t)=(\mathcal{P}\otimes\mathbb{P}_\nu)(H_t^{Q_t}\subseteq V_0)=\mathcal{E}[(1-p)^{|H_t^{Q_t}|}]\geq (1-p)^{\mathcal{E}[|H_t^{Q_t}|]}
\end{align}
using Jensen's inequality, where $V_0=\left\{x\in\mathbb{Z}^d: \xi(x,0)=0\right\}$ denotes the initial set of the vacancies in the exclusion process. Let $\tau_x$ and $\tau_{Q_t}$ denote the hitting time of the point $x$ and the set $Q_t$, respectively, by a simple random walk with jump rate $2d\rho$, and write $\mathbb{P}_y$ for the probability with respect to this walk when starting in $y$. Then,
\begin{align*}
\mathcal{E}[|H_t^{Q_t}|]&=\sum_{x\in\mathbb{Z}^d}\mathcal{P}(\exists \text{ path in } \mathcal{G} \text{ from } (0,x) \text{ to } [0,t]\times Q_t)
\\&=\sum_{x\in\mathbb{Z}^d}\mathbb{P}_x(\tau_{Q_t}\leq t)=|Q_t|+\sum_{x\notin Q_t}\mathbb{P}_x(\tau_{Q_t}\leq t).
\end{align*}
On the other hand, for every $\varepsilon>0$ we have
\begin{align*}
\sum_{x\in\mathbb{Z}^d}\mathbb{P}_x(\tau_0\leq t+\varepsilon t)&\geq \sum_{x\in\mathbb{Z}^d}\mathbb{P}_x(\tau_0\leq t+\varepsilon t, \tau_{Q_t}\leq t)\geq \sum_{x\notin Q_t}\mathbb{P}_x(\tau_0\leq t+\varepsilon+t, \tau_{Q_t}\leq t)
\\&\geq \inf_{z\in\partial Q_t}\mathbb{P}_z(\tau_0\leq \varepsilon t)\sum_{x\notin Q_t}\mathbb{P}_x(\tau_{Q_t}\leq t),
\end{align*}
where in the last step we used the Markov property. Hence,
\begin{align*}
\sum_{x\notin Q_t}\mathbb{P}_x(\tau_{Q_t}\leq t)\leq \frac{\sum_{x\in\mathbb{Z}^d}\mathbb{P}_x(\tau_0\leq t+\varepsilon t)}{\inf_{z\in\partial Q_t}\mathbb{P}_z(\tau_0\leq \varepsilon t)}
\end{align*}
and therefore
\begin{align}\label{ineqh}
\mathcal{E}[|H_t^{Q_t}|]\leq |Q_t|+\frac{\sum_{x\in\mathbb{Z}^d}\mathbb{P}_x(\tau_0\leq t+\varepsilon t)}{\inf_{z\in\partial Q_t}\mathbb{P}_z(\tau_0\leq \varepsilon t)}.
\end{align}
In the proof of \cite[Lemma 2.1]{drewitz} it was shown that ${\inf_{z\in\partial Q_t}\mathbb{P}_z(\tau_0\leq \varepsilon t)}\to 1$ as $t\to\infty$ for both choices of $r_t\in\{\sqrt{t/\log(t)}, \log(t)\}$ in the respective dimension. Thus, for $t\to\infty$ and $\varepsilon\to 0$ we can replace the right hand-side of \eqref{ineqh} by
\begin{align}\label{rhs}
|Q_t|+\sum_{x\in\mathbb{Z}^d}\mathbb{P}_x(\tau_0\leq t)=|Q_t|+\sum_{x\in\mathbb{Z}^d}\mathbb{P}_0(\tau_x\leq t)=|Q_t|+\mathbb{E}_0[R_t],
\end{align}
where we denote by $R_t$ the range of a simple random walk with jump rate $2d\rho$, which is known (cf.\,\cite{arratia}) to have the asymptotics
\begin{align}\label{Rt}
\mathbb{E}_0[R_t]=\left\{\begin{array}{ll}\displaystyle 4\sqrt{\frac{\rho}{\pi}}\sqrt{t}(1+o(1)), &d=1,\\[12pt]\displaystyle 4\rho\pi\frac{t}{\log(t)}(1+o(1)), &d=2,
\end{array}\right.
\end{align}
as $t\to\infty$. Combining \eqref{asA}, \eqref{rhs} and \eqref{Rt} we deduce that
\begin{align*}
\log\mathbb{P}(A_t)\geq \log(1-p)^{|Q_t|+\mathbb{E}_0[R_t]}=\left\{\begin{array}{ll}\displaystyle\log(1-p)\left(1+2\sqrt{\frac{t}{\log(t)}}+4\sqrt{\frac{\rho}{\pi}}\sqrt{t}(1+o(1))\right), &d=1,\\[12pt]\displaystyle \log(1-p)\left(1+2\log(t)+4\rho\pi\frac{t}{\log(t)}(1+o(1))\right), &d=2.
\end{array}\right.
\end{align*}
Hence,
\begin{align*}
\lim_{t\to\infty}\frac{1}{\sqrt{t}}\log\mathbb{P}(A_t)\geq 4\log(1-p)\sqrt{\frac{\rho}{\pi}}, \qquad d=1,
\end{align*}
and
\begin{align*}
\lim_{t\to\infty}\frac{\log(t)}{t}\log\mathbb{P}(A_t)\geq 4\log(1-p)\rho\pi, \qquad d=2,
\end{align*}
which proves the lemma. 
\end{proof}

\section{Survival probability with Dormancy}
In this section, we establish Theorem 1.2 under the assumption of the stochastic dormancy mechanism. In particular, we assume that $s_0,s_1>0$. Similar to the proof of Theorem 1.1, the argument is divided into two lemmas, which provide the upper and lower bounds, respectively. While the proof ideas are analogous, they additionally account for the behaviour of the switching mechanism $\alpha$.
\begin{lemma}[Upper bound]
For all $\gamma\in(-\infty, 0)$,
\begin{align}\label{ass41}
\log\mathbb{E}_{(0,1)}^{(X,\alpha)}\mathbb{E}_\nu^\xi\left[\displaystyle\e^{\displaystyle\gamma\int_0^t\alpha(s)\xi(X(s),s)\,\d s}\right]\leq \left\{\begin{array}{ll}\displaystyle-4p\sqrt{\frac{s_0\rho}{(s_0+s_1)\pi}}\sqrt{t}(1+o(1)), &d=1,\\[13pt]\displaystyle-4p\frac{s_0\rho\pi}{s_0+s_1}\frac{t}{\log(t)}(1+o(1)), &d=2,\\[13pt]\displaystyle-\tilde{\lambda}_{d,\gamma,\rho,p,s_0,s_1}t(1+o(1)), &d\geq 3,\end{array}\right.
\end{align}
as $t\to\infty$, for some constant $\tilde{\lambda}_{d,\gamma,\rho,p,s_0,s_1}$ satisfying 
\begin{align*}
\tilde{\lambda}_{d,\gamma,\rho,p,s_0,s_1}\geq \inf_{a\in[0,1]}\left\{s_0-\sqrt{s_0s_1}+\left(s_1-s_0+\frac{p\rho}{\frac{\rho}{|\gamma|}+G_d(0)}\right)a-2\sqrt{s_0s_1a(1-a)}\right\}.
\end{align*}

\end{lemma}
\begin{proof}
The proof is similar to the proof of the upper bound without dormancy. Fix a realization of $(X,\alpha)$ and let $\tau_1,\tau_2,\cdots,\tau_{N_t}$ be the jump times of the piecewise constant function $X$ (given $\alpha$) up to time $t$ and $x_k$, $k=0,\cdots,N_t$ the corresponding values on each intervall $[\tau_k,\tau_{k+1})$ where we set $x_0=0$ and $\tau_0=0$. Define the function $K:\mathbb{Z}^d\times[0,t]\to\mathbb{R}$ as
\begin{align*}
K(z,s)=\gamma\alpha(s)\delta_{x_k}(z)\quad \text{for all } s\in[\tau_k,\tau_{k+1}), k=1,\cdots,n.
\end{align*}
Then we have $K\leq 0$ and
\begin{align*}
\sum_{z\in\mathbb{Z}^d}\int_0^t|K(z,s)|\,\d s=|\gamma |L_t(1)<\infty
\end{align*}
for all $t\geq 0$. Thus, by Lemma \ref{comparison},
\begin{align*}
\mathbb{E}_\nu^\xi\left[\exp\left(\gamma\int_0^t\alpha(s)\xi(X(s),s)\,\d s\right)\right]\leq \mathbb{E}_\nu^{\hat{\xi}}\left[\exp\left(\gamma\int_0^t\alpha(s)\hat{\xi}(X(s),s)\,\d s\right)\right].
\end{align*}
Using the same notation as in the proof of Lemma 3.1 and integrating over the Bernoulli system of independent random walks again as well as the fact that $\nu$ is a Bernoulli product measure, we obtain
\begin{align}\label{ineqXa}
\nonumber\mathbb{E}_\nu^{\hat{\xi}}\left[\exp\left(\gamma\int_0^t\alpha(s)\hat{\xi}(X(s),s)\,\d s\right)\right]&=\int_{\{0,1\}^{\mathbb{Z}^d}}\prod_{y\in A_\eta}v_{(X,\alpha)}(y,t)\nu(\d\eta)
\\&=\prod_{y\in\mathbb{Z}^d}p v_{(X,\alpha)}(y,t)\leq \exp\left(p\sum_{y\in\mathbb{Z}^d}(v_{(X,\alpha)}(y,t)-1)\right),
\end{align}
where we wrote
\begin{align*}
v_{(X,\alpha)}(y,t):=\mathbb{E}_y^{Y_y}\left[\exp\left(\gamma\int_0^t\delta_{(X(s),\alpha(s))}(Y_y(s),1)\,\d s\right)\right]
\end{align*}
for a fixed realization of $(X,\alpha)$.
In \cite{shafigh} it has been shown that the right hand-side of \eqref{ineqXa} is nothing but the survival probability of the switching random walk $X$ among a Poisson system of moving traps, which is maximized for $\kappa=0$ and follows the same asymptotics as the right hand-side of \eqref{ass41}. Since this bound holds for any realization of $(X,\alpha)$, the lemma follows.
\end{proof}
We now prove the corresponding lower bound on the survival probability. 
\begin{lemma}[Lower bound]\label{lowerboundsurvival}
For all $\gamma\in(-\infty,0)$ we have
\begin{align*}
\log\mathbb{E}_{(0,1)}^{(X,\alpha)}\mathbb{E}_\nu\left[\exp\left(\gamma\int_0^t\xi(X(s),s)\,\d s\right)\right]\geq \left\{\begin{array}{ll}\displaystyle 4\log(1-p)\sqrt{\frac{s_0\rho}{(s_0+s_1)\pi}}\sqrt{t}(1+o(1)), &d=1,\\[13pt]\displaystyle 4\log(1-p)\frac{s_0\rho\pi}{s_0+s_1}\frac{t}{\log(t)}(1+o(1)), &d=2,\end{array}\right.
\end{align*}
as $t\to\infty$.
\end{lemma}
\begin{proof}
The proof follows a similar approach to that of Lemma \ref{lowerwithout}. Define $Q_t:=[-r_t,r_t]^d$, where we set $r_t=\sqrt{\frac{t}{\log(t)}}$ for $d=1$ and $r_t=\log(t)$ for $d=2$. For a switching random walk, the strategy is to constrain the random walk $X$ to stay in $Q_t$ up to time $t$ and to require the exclusion process to create a vacant region $Q_t$ around zero at time $t=0$, while maintaining this vacancy during all time intervals in which the random walk $X$ is active, i.\,e.\,, the switching component $\alpha$ takes the value $1$. In other words, when the random walk is dormant, exclusion particles may enter the vacant region, but they must leave before the random walk becomes active again. More precisely, define $A_t:=\left\{s\in[0,t]:\alpha(s)=1\right\}$ as well as
\begin{align*}
B_t:=\left\{X(s)\in Q_t\forall s\in[0,t]\right\}\qquad\text{ and }\qquad C_t:=\left\{\xi(x,s)=0\forall x\in Q_t\forall s\in A_t\right\}
\end{align*}
for a fixed realization of $\alpha$ . Then,
\begin{align*}
\mathbb{E}_0^X\mathbb{E}_\nu^\xi\left[\exp\left(\gamma\int_0^t\xi(X(s),s)\,\d s\right)\right]\geq \mathbb{P}_{(0,1)}^{(X,\alpha)}(B_t)\mathbb{P}_1^{\alpha}\otimes\mathbb{P}_\nu^\xi(C_t).
\end{align*}
Now, if $\tilde{X}$ denotes a simple symmetric random walk without switching and with jump rate $2d\kappa$, then end-point $X(t)$ equals $\tilde{X}(L_t(1))$ in distribution, such that we can write 
\begin{align*}
\mathbb{P}_{(0,1)}^{(X,\alpha)}(B_t)&=\mathbb{P}_1^{\alpha}\mathbb{P}_0^{\tilde{X}}(\tilde{X}(L_s(1))\in Q_t\forall s\in[0,t])\geq \mathbb{P}_0^{\tilde{X}}(\tilde{X}(s)\in Q_t\forall s\in[0,t])
\end{align*}
and hence
\begin{align*}
\mathbb{P}_{(0,1)}^{(X,\alpha)}(B_t)\geq \left\{\begin{array}{ll}\exp\left(\log(\beta)\log(t)\right), &d=1,\\\exp\left(\log(\beta)\frac{t}{\log(t)^2}\right), &d=2.
\end{array}\right.
\end{align*}
as in \eqref{asB}. In order to determine $\mathbb{P}_1^{\alpha}\mathbb{P}_\nu^\xi(C_t)$, we once more make use of the graphical representation of the exclusion process. Denote by
\begin{align*}
\tilde{H}_t^{Q_t}:=\left\{x\in\mathbb{Z}^d : \exists \text{ path in } \mathcal{G} \text{ from } (0,x) \text{ to } A_t\times Q_t\right\} 
\end{align*}
the set of the starting point of all paths of the exclusion process which enter the set $Q_t$ at some time point $s\in A_t$, and define $\tilde{H}_t^0$ analogously with $Q_t$ replaced by $\{0\}$. Note, that $\tilde{H}_t^{Q_t}\subseteq H_t^{Q_t}$. Then, as in the proof of Lemma \ref{lowerwithout}, we see that
\begin{align*}
\mathbb{P}_\nu^\xi(C_t)=(\mathcal{P}\otimes\mathbb{P}_\nu^\xi)(\tilde{H}_t^{Q_t}\subseteq V_0),
\end{align*}
since otherwise there would be some exclusion particle propagating into $Q_t$ at some time point $s\in A_t$. Thus,
\begin{align*}
\mathbb{P}_\nu^\xi(C_t)=\mathcal{E}[(1-p)^{|\tilde{H}_t^{Q_t}|}]\geq (1-p)^{\mathcal{E}[|\tilde{H}_t^{Q_t}|]}
\end{align*}
using Jensen's inequality again, where
\begin{align}\label{Etilde}
\mathcal{E}[|\tilde{H}_t^{Q_t}|]&=\sum_{x\in\mathbb{Z}^d}\mathcal{P}(\exists \text{ path in } \mathcal{G} \text{ from } (0,x) \text{ to } A_t\times Q_t).
\end{align}
For the fixed realization of $\alpha$, let $N_t$ denote the number of jumps of $\alpha$ till $t$ and write $s_1,s_2,\cdots,s_{N_t}$ for the jump times of $\alpha$ and $\tau_k:=s_k-s_{k-1}$ for the corresponding waiting times, such that $A_t=[0,s_1)\cup[s_2,s_3)\cup\cdots\cup[s_{N_t},t]$, where we w.l.o.g.\,assume that $N_t$ is even. We will show that
\begin{align*}
\Psi(t):=\mathcal{P}(\exists \text{ path in } \mathcal{G} \text{ from } (0,x) \text{ to } Q_t\times A_t)=\mathcal{P}(\exists \text{ path in } \mathcal{G} \text{ from } (0,x) \text{ to } Q_t\times [0,L_t(1)]).
\end{align*}
To this end, note that, if $\tau_{E}$ denotes the hitting time of a set $E$ by a simple random walk with jump rate $2d\rho$, then
\begin{align*}
\mathcal{P}(\exists \text{ path in } \mathcal{G} \text{ from } (0,x) \text{ to } Q_t\times [s_{2k},s_{2k+1}))=\mathbb{P}_x(\tau_{Q_t}\in[s_{2k},s_{2k+1}))
\end{align*}
for $k=0,\cdots,\frac{1}{2}N_t$ and hence
\begin{align*}
\Psi(t)=\int_0^{s_1}\mathbb{P}_x(\tau_{Q_t}=s)\,\d s+\int_{s_2}^{s_3}\mathbb{P}_x(\tau_{Q_t}=s)\,\d s+\cdots+\int_{s_{N_t}}^t \mathbb{P}_x(\tau_{Q_t}=s)\,\d s.
\end{align*}
Now, we substitute the integration variable in a similar manner as in the proof of \cite[Lemma 2.8]{shafigh} to observe that
\begin{align*}
\Psi(t)=\mathbb{P}_x(\tau_{Q_t}\in[0,L_t(1)])=\mathcal{P}(\exists \text{ path in } \mathcal{G} \text{ from } (0,x) \text{ to } Q_t\times [0,L_t(1)]).
\end{align*}
Combining this with \eqref{Etilde} we obtain
\begin{align*}
\mathcal{E}[|\tilde{H}_t^{Q_t}|]=\sum_{x\in\mathbb{Z}^d}\mathbb{P}_x(\tau_{Q_t}\leq L_t(1))=\mathcal{E}[|H_{L_t(1)}^{Q_t}|],
\end{align*}
such that we can apply the results and details of the proof of Lemma \ref{lowerwithout} to deduce that
\begin{align*}
\mathbb{P}_\nu^\xi(C_t)\geq (1-p)^{\mathcal{E}[|H_{L_t(1)}^{Q_t}|]}=(1-p)^{|Q_t|+\mathbb{E}_0[R_{L_t(1)}]}.
\end{align*}
and therefore
\begin{align*}
\mathbb{P}_1^{\alpha}\mathbb{P}_\nu^\xi(C_t)\geq\displaystyle \mathbb{E}_1^{\alpha}\left[\e^{\log(1-p)\left(|Q_t|+\mathbb{E}_0[R_{L_t(1)}]\right)}\right].
\end{align*}
Note that
\begin{align*}
\mathbb{E}_0[R_{L_t(1)}]=\left\{\begin{array}{ll}\displaystyle 4\sqrt{\frac{\rho}{\pi}}\sqrt{L_t(1)}(1+o(1)), &d=1,\\[12pt]\displaystyle 4\rho\pi\frac{L_t(1)}{\log(L_t(1))}(1+o(1)), &d=2,
\end{array}\right.
\end{align*}
such that $\mathbb{E}_0[R_t]$ grows on a slower scale than the large deviation scale $t$ of the local times of $\alpha$. Applying \cite[Lemma 2.4]{shafigh} as stated in \eqref{lemma2.4} to $f(t)=\sqrt{t}$ for $d=1$ and $f(t)=t/\log(t)$ in $d=2$, we obtain
\begin{align*}
\lim_{t\to\infty}\frac{1}{\sqrt{t}}\log\mathbb{P}_1^{\alpha}\mathbb{P}_\nu^\xi(C_t)\geq 4\log(1-p)\sqrt{\frac{\rho}{\pi}}\sqrt{\frac{s_0}{s_0+s_1}}
\end{align*}
for $d=1$ and
\begin{align*}
\lim_{t\to\infty}\frac{\log(t)}{t}\log\mathbb{P}_1^{\alpha}\mathbb{P}_\nu^\xi(C_t)\geq 4\log(1-p)\rho\pi\frac{s_0}{s_0+s_1}
\end{align*}
with the same argument regarding the term $|Q_t|$ as in the proof of Lemma \ref{lowerwithout}.
\end{proof}

\section{Growth with Dormancy}
This section is devoted to the case $\gamma>0$, where the exclusion particles are interpreted as catalysts. We start with the proof of the variational formula \eqref{varfor}:
\\\\\textbf{\emph{Proof of Theorem 1.3}}
\\As stated in Section 2.4, our goal is to apply the Perron-Frobenius theory for self-adjoint operators, even though the generator $Q$ of $\alpha$ is not symmetric. To address this issue, we will instead work with the symmetric version $\tilde{Q}$ defined \eqref{deftildeq}, by employing the Radon-Nikodym derivative \eqref{rndichte}. To this end, define
\begin{align}\label{V}
V(\eta,x,i):=-s_0\delta_0(i)-s_1\delta_1(i)+\gamma\delta_{(1,1)}(i,\eta(x))
\end{align}
for $(\eta,x,i,)\in\{0,1\}^{\mathbb{Z}^d}\times\mathbb{Z}^d\times\{0,1\}$. Then, 
\begin{align*}
\mathbb{E}_{\nu,(0,1)}^{(\xi,X,\alpha)}\left[\exp\left(\int_0^t\delta_{(1,1)}(\alpha(s),\xi(X(s),s))\d s\right)\right]=\e^{\sqrt{s_0s_1}t}\mathbb{E}_{\nu,(0,1)}^{(\tilde{\xi},\tilde{X},\tilde{\alpha})}\left[\exp\left(\int_0^tV(\tilde{\xi}(s),\tilde{X}(s),\tilde{\alpha}(s))\d s\right)\right].
\end{align*}
Let us start with the proof of the upper bound, which is done in a standard way. More precisley, for $Y(s):=(\tilde{\xi}(s),\tilde{X}(s),\tilde{\alpha}(s))$, we have
\begin{align*}
\mathbb{E}_{\nu,(0,1)}^{(\tilde{\xi},\tilde{X},\tilde{\alpha})}\left[\exp\left(\int_0^tV(Y(s))\d s\right)\right]=\mathbb{E}_{\nu,(0,1)}^{(\tilde{\xi},\tilde{X},\tilde{\alpha})}\left[\exp\left(\int_0^tV(Y(s))\d s\right)\mathds{1}_{\{\tilde{X}(t)\in Q_{r(t)}\}}\right]+R_t
\end{align*}
with $r(t):=t\log(t)$, $Q_{r(t)}=[-r(t),r(t)]^d$ and $R_t$ some function with $\lim_{t\to\infty}\frac{1}{t}\log R_t=-\infty$ by some standard large deviation estimates for the random walk. Moreover,
\begin{align*}
\mathbb{E}_{\nu,(0,1)}^{(\xi,\tilde{X},\tilde{\alpha})}\left[\exp\left(\int_0^tV(Y(s))\d s\right)\mathds{1}_{\{\tilde{X}(t)\in Q_{r(t)}\}}\right]&\leq \sum_{x\in Q_{r(t)}}\mathbb{E}_{\nu,(x,1)}^{(\xi,\tilde{X},\tilde{\alpha})}\left[\exp\left(\int_0^tV(Y(s))\d s\right)\mathds{1}_{\{\tilde{X}(t)\in Q_{r(t)}\}}\right]
\\&=(1+o(1))\left(\e^{(\tilde{L}+V)t}\mathds{1}_{Q_{r(t)}},\mathds{1}_{Q_{r(t)}}\right)
\\&\leq(1+o(1))\e^{\lambda t}\|Q_{r(t)}\|^2\leq (1+o(1))\e^{\lambda t}|Q_{r(t)}|,
\end{align*}
where 
\begin{align*}
\lambda:=\sup\text{Sp}(\tilde{L}+V)
\end{align*}
denotes the largest eigenvalue of the bounded and self-adjoint operator $\tilde{L}+V$. Since $|Q_{r(t)}|=(2t\log(t))^d$ grows only polynomially, we obtain
\begin{align*}
\lim_{t\to\infty}\frac{1}{t}\log\mathbb{E}_{\nu,(0,1)}^{(\tilde{\xi},\tilde{X},\tilde{\alpha})}\left[\exp\left(\int_0^tV(Y(s))\d s\right)\right]\leq \lambda.
\end{align*}
The proof of the lower bound follows the same approach as the proof of \cite[Proposition 2.2.1]{exclusion}. More precisely, let $(E_\mu)_{\mu\in\mathbb{R}}$ denote the spectral family of orthogonal projection operators associated with $\tilde{L}+V$, and let $\delta>0$. Then we can find a function $f_{\delta}\in\ell^2\left(\{0,1\}^{\mathbb{Z}^d}\times\mathbb{Z}^d\times\{0,1\}\right)$ such that $(E_\lambda-E_{\lambda-\delta})f_\delta\neq 0$. Approximating $f_\delta$ by bounded functions with finite support in the spatial component, we can w.l.o.g. assume that $0\leq f_\delta\leq \mathds{1}_{K_\delta}$ for some finite $K_\delta\subseteq\mathbb{Z}^d$. Then,
\begin{align*}
\mathbb{E}_{\nu,(0,1)}^{(\tilde{\xi},\tilde{X},\tilde{\alpha})}\left[\exp\left(\int_0^tV(Y(s))\d s\right)\right]&\geq \sum_{x\in K_\delta}\mathbb{E}_{\nu,(0,1)}^{(\tilde{\xi},\tilde{X},\tilde{\alpha})}\left[\exp\left(\int_1^tV(Y(s))\d s\right)\mathds{1}_{\{\tilde{X}(1)=x\}}\right]
\\&\geq \sum_{x\in K_\delta}\mathbb{E}_{\nu,(0,1)}^{(\tilde{\xi},\tilde{X},\tilde{\alpha})}\left[\mathbb{E}_{\eta(1),(x,\alpha(1))}^{(\xi,\tilde{X},\tilde{\alpha})}\left[\exp\left(\int_0^{t-1}V(Y(s))\d s\right)\right]\mathds{1}_{\{\tilde{X}(1)=x\}}\right]
\\&=\sum_{x\in K_\delta}\sum_{i\in\{0,1\}}p_1(x,i)\mathbb{E}_{\nu,(x,i)}^{(\tilde{\xi},\tilde{X},\tilde{\alpha})}\left[\exp\left(\int_0^{t-1}V(Y(s))\d s\right)\right]
\\&\geq \sum_{x\in K_\delta}p_1(x,1)\mathbb{E}_{\nu,(x,1)}^{(\tilde{\xi},\tilde{X},\tilde{\alpha})}\left[\exp\left(\int_0^{t-1}V(Y(s))\d s\right)\right]
\end{align*}
with $p_t(x,i):=\mathbb{P}_{(0,1)}^{(\tilde{X},\tilde{\alpha})}((\tilde{X}(t),\tilde{\alpha}(t))=(x,i))$, where we used that $\nu$ is invariant under the exclusion dynamics. Then, defining $C_\delta:=\min_{x\in K_\delta}p_1(x,1)$, we can lower-bound
\begin{align*}
\mathbb{E}_{\nu,(0,1)}^{(\tilde{\xi},\tilde{X},\tilde{\alpha})}&\left[\exp\left(\int_0^tV(Y(s))\d s\right)\right]
\\&\geq C_\delta\sum_{x\in K_\delta}\mathbb{E}_{\nu,(x,1)}^{(\xi,\tilde{X},\tilde{\alpha})}\left[\exp\left(\int_0^{t-1}V(Y(s))\d s\right)\right]
\\&\geq C_\delta\int_{\{0,1\}^{\mathbb{Z}^d}}\nu(\d\eta)\sum_{(x,i)\in\mathbb{Z}^d\times\{0,1\}}f_\delta(\eta,x,i)\mathbb{E}_{\eta,(x,i)}^{(\xi,\tilde{X},\tilde{\alpha})}\left[\exp\left(\int_0^{t-1}V(Y(s))\d s\right)f_\delta(Y(t-1))\right]
\\&\geq C_\delta\left(\e^{(\tilde{L}+V)(t-1)}f_\delta,f_\delta\right) \geq C_\delta\e^{(\lambda-\delta)(t-1)}\|(E_\lambda-E_{\lambda-\delta})f_\delta\|^2.
\end{align*}
This yields 
\begin{align*}
\liminf_{t\to\infty}\frac{1}{t}\log \mathbb{E}_{\nu,(0,1)}^{(\xi,\tilde{X},\tilde{\alpha})}&\left[\exp\left(\int_0^tV(Y(s))\d s\right)\right]\geq \lambda-\delta
\end{align*}
and hence, letting $\delta\to 0$ completes the proof of the lower bound. Moreover, the Rayleigh-Ritz formula asserts that
\begin{align*}
\sup\text{Sp}(\tilde{L}+V)=\sup_{\substack{f\in\ell^2\left(\{0,1\}^{\mathbb{Z}^d}\times\mathbb{Z}^d\times\{0,1\}\right), \\\|f\|_2=1}}\left<(\tilde{L}+V)f,f\right>.
\end{align*} 
Let us calculate the inner product. We have
\begin{align*}
\left<Vf,f\right>=\int_{\{0,1\}^{\mathbb{Z}^d}}\nu(\d\eta)\sum_{x\in\mathbb{Z}^d}(\gamma\eta(x)-s_1)f(\eta,x,1)^2-s_0f(\eta,x,0)^2,
\end{align*}
and
\begin{align*}
\left<\tilde{L}f,f\right>=-A_2(f)-A_3(f)-A_4(f),
\end{align*}
where
\begin{align*}
A_2(f):&=-\int_{\{0,1\}^{\mathbb{Z}^d}}\nu(\d\eta)\sum_{z\in\mathbb{Z}^d}\sum_{i\in\{0,1\}}\rho\sum_{x\in\mathbb{Z}^d}\sum_{\substack{y\in\mathbb{Z}^d\\x\sim y}}(f(\eta^{x,y},z,i)-f(\eta,z,i))f(\eta,z,i)
\\&=\frac{1}{2}\rho\int_{\{0,1\}^{\mathbb{Z}^d}}\nu(\d\eta)\sum_{z\in\mathbb{Z}^d}\sum_{i\in\{0,1\}}\sum_{\substack{x,y\in\mathbb{Z}^d\\x\sim y}}(f(\eta^{x,y},z,i)-f(\eta,z,i))^2,
\end{align*}
and, writing $e_j\in\mathbb{Z}^d$ for the $j$-th unit vector,
\begin{align*}
A_3(f):&=-\int_{\{0,1\}^{\mathbb{Z}^d}}\nu(\d\eta)\sum_{z\in\mathbb{Z}^d}\sum_{i\in\{0,1\}}i\kappa\sum_{y\sim z}i\kappa(f(\eta,y,i)-f(\eta,z,i))f(\eta,z,i)
\\&=-\int_{\{0,1\}^{\mathbb{Z}^d}}\nu(\d\eta)\sum_{z\in\mathbb{Z}^d}\sum_{j=1}^d\kappa(f(\eta,z+e_j,1)-f(\eta,z,1))f(\eta,z,1)
\\&\quad-\int_{\{0,1\}^{\mathbb{Z}^d}}\nu(\d\eta)\sum_{z\in\mathbb{Z}^d}\sum_{j=1}^d\kappa(f(\eta,z,1)-f(\eta,z+e_j,1))f(\eta,z+e_j,1)
\\&=\int_{\{0,1\}^{\mathbb{Z}^d}}\nu(\d\eta)\sum_{z\in\mathbb{Z}^d}\frac{1}{2}\sum_{y\in\mathbb{Z}^d,y\sim z}\kappa(f(\eta,y,1)-f(\eta,z,1))^2.
\end{align*}
Finally, we have
\begin{align*}
A_4(f):&=-\int_{\{0,1\}^{\mathbb{Z}^d}}\nu(\d\eta)\sum_{z\in\mathbb{Z}^d}\sum_{i\in\{0,1\}}\sqrt{s_0s_1}(f(\eta,z,1-i)-f(\eta,z,i))f(\eta,z,i)
\\&=\int_{\{0,1\}^{\mathbb{Z}^d}}\nu(\d\eta)\sum_{z\in\mathbb{Z}^d}2\sqrt{s_0s_1}(f(\eta,z,1)-f(\eta,z,0))^2,
\end{align*}
which finishes the proof. \hfill$\Box$

Next, we proceed to prove Theorem 1.4, addressing the recurrent case $d\in\{1,2\}$ and the transient case $d\geq 3$ separately, as each requires a distinct approach.
\\\\\textbf{\emph{Proof of Theorem 1.4(a)}}
For the upper bound, note that
\begin{align*}
\mathbb{E}_{(0,1)}^{(X,\alpha)}\mathbb{E}_\nu^\xi\left[\exp\left(\gamma\int_0^t\alpha(s)\xi(X(s),s)\,\d s\right)\right]&\leq \mathbb{E}_{(0,1)}^{(X,\alpha)}\left[\exp\left(\gamma\int_0^t\alpha(s)\,\d s\right)\right]
\\&=\mathbb{E}_1^\alpha\left[\exp\left(\gamma L_t(1)\right)\right]
\end{align*}
and hence
\begin{align*}
\lim_{t\to\infty}\frac{1}{t}\log\mathbb{E}_{(0,1)}^{(X,\alpha)}\mathbb{E}_\nu^\xi\left[\exp\left(\gamma\int_0^t\alpha(s)\xi(X(s),s)\,\d s\right)\right]\leq \sup_{a\in[0,1]}\left\{a\gamma-I(a)\right\}
\end{align*}
using Varadhan's lemma and the large deviation principle for $\alpha$. The lower bound relies on a similar idea to the proof of the lower bound in \cite[Theorem 1.3.2(i)]{exclusion} in the case without switching. Specifically, we constrain the random walk to stay within a finite set $Q$ around zero up to time $t$ and require the exclusion process to build an area full of particles in $Q$ up to time $t$. More precisely, 
\begin{align*}
\mathbb{E}_{(0,1)}^{(X,\alpha)}\mathbb{E}_\nu^\xi\left[\exp\left(\gamma\int_0^t\alpha(s)\xi(X(s),s)\,\d s\right)\right]&\geq \mathbb{E}_1^{\alpha}[\exp(\gamma L_t(1))]\mathbb{P}_{(0,1)}^{(X,\alpha)}(X(s)\in Q\forall s\in[0,t])\phi(t)
\\&\geq \mathbb{E}_1^{\alpha}[\exp(\gamma L_t(1))]\mathbb{P}_0^{\tilde{X}}(\tilde{X}(s)\in Q\forall x\in[0,t])\phi(t)
\end{align*}
for a simple symmetric random walk $\tilde{X}$ with jump rate $2d\kappa$ without switching and where we recall $\phi(t)=\mathbb{P}_\nu^\xi(\xi(0,s)=0\forall s\in[0,t])$. Now, \cite[Lemma 3.1.1]{exclusion} asserts that
\begin{align*}
\lim_{t\to\infty}\frac{1}{t}\log\phi(t)=0.
\end{align*}
Moreover, 
\begin{align*}
\lim_{t\to\infty}\frac{1}{t}\log\mathbb{P}_0^{\tilde{X}}(\tilde{X}(s)\in Q\forall x\in[0,t])=-\lambda(Q)
\end{align*}
for $\lambda$ the Dirichlet eigenvalue of $-\kappa\Delta$ on $Q$. Letting $Q\to\mathbb{Z}^d$ and therefore $\lambda\to 0$, we deduce that
\begin{align*}
\lim_{t\to\infty}\frac{1}{t}\log\mathbb{E}_{(0,1)}^{(X,\alpha)}\mathbb{E}_\nu^\xi\left[\exp\left(\gamma\int_0^t\alpha(s)\xi(X(s),s)\,\d s\right)\right]\geq \sup_{a\in[0,1]}\left\{a\gamma-I(a)\right\}
\end{align*}
using Varadhan's lemma again. Calculating the maximal value in the brackets yields the assertion. \hfill $\Box$

The following proof is based on the variational formula \eqref{varfor} and follows a similar idea as in the proof of \cite[Proposition 3.2.1]{exclusion}:
\\\\\textbf{\emph{Proof of Theorem 1.4(b)}}
\\Let $\varepsilon>0$, and $a_0,a_1\in[0,1]$ be constants to be determined later, such that $a_0+a_1=1$. Choose a function $\phi_\varepsilon:\mathbb{Z}^d\times\{0,1\}\to\mathbb{R}$ which satisfies the conditions
\begin{align}
\sum_{x\in\mathbb{Z}^d}\phi_\varepsilon(x,i)^2=a_i, \quad i\in\{0,1\},
\end{align}
and
\begin{align}
\max\left\{\sum_{i\in\{0,1\}}\sum_{\substack{x,y\\x\sim y}}(\phi_\varepsilon(x,i)-\phi_\varepsilon(y,i))^2, \sum_{x\in\mathbb{Z}^d}(\phi_\varepsilon(x,i)-\phi_\varepsilon(x,1-i))^2\right\}\leq\varepsilon^2.
\end{align}
Then, for $f_\varepsilon:\{0,1\}^{\mathbb{Z}^d}\times\mathbb{Z}^d\times\{0,1\}\to\mathbb{R}$ defined as
\begin{align*}
f_\varepsilon(\eta,x,i):=\frac{1+\varepsilon\eta(x)}{\sqrt{1+(2\varepsilon+\varepsilon^2)p}}\phi_\varepsilon(x,i),
\end{align*}
we have
\begin{align*}
\|f_\varepsilon\|&=\int_{\{0,1\}^{\mathbb{Z}^d}}\nu(\d\eta)\sum_{x\in\mathbb{Z}^d}\sum_{i\in\{0,1\}}f_\varepsilon(\eta,x,i)^2
\\&=\int_{x\in\mathbb{Z}^d}\sum_{i\in\{0,1\}}\left(p\frac{(1+\varepsilon)^2}{1+(2\varepsilon+\varepsilon^2)p}+(1-p)\frac{1}{1+(2\varepsilon+\varepsilon^2)p}\right)\phi_\varepsilon(x,i)^2=1.
\end{align*}
Hence, $f_\varepsilon$ can serve as a test function  for \eqref{varfor}. Note that
\begin{align*}
I:&=\int_{\{0,1\}^{\mathbb{Z}^d}}\nu(\d\eta)\sum_{z\in\mathbb{Z}^d}\sum_{i\in\{0,1\}}(i\gamma\nu(\d\eta)-s_i)f_\varepsilon(\eta,z,i)^2
\\&=\sum_{x\in\mathbb{Z}^d}\frac{(-s_0p(1+\varepsilon)^2-s_0(1-p))\phi_\varepsilon(x,0)^2+(p(\gamma-s_1)(1+\varepsilon^2)-s_1(1-p))\phi_\varepsilon(x,1)^2}{1+(2\varepsilon+\varepsilon^2)p}
\\&=\left(\frac{\gamma(p+(2\varepsilon+\varepsilon^2)p)}{1+(2\varepsilon+\varepsilon^2)p}-s_1\right)\sum_{x\in\mathbb{Z}^d}\phi_\varepsilon(x,1)^2-s_0\sum_{x\in\mathbb{Z}^d}\phi_\varepsilon(x,0)^2
\\&=a_1\left(p\frac{\gamma +2\varepsilon+\varepsilon^2}{1+(2\varepsilon+\varepsilon^2)p}-s_1+s_0\right)-s_0.
\end{align*}
Analogously, we calculate the other terms appearing in the right hand-side of \eqref{varfor} to see that
\begin{align*}
II:&=\int_{\{0,1\}^{\mathbb{Z}^d}}\nu(\d\eta)\sum_{z\in\mathbb{Z}^d}\sum_{i\in\{0,1\}}\frac{1}{2}\sum_{x,y\in\mathbb{Z}^d, x\sim y}\rho(f_\varepsilon(\eta^{x,y},z,i)-f_\varepsilon(\eta,z,i))^2
\\&=\frac{1}{1+(2\varepsilon+\varepsilon^2)p}\int_{\{0,1\}^{\mathbb{Z}^d}}\nu(\d\eta)\frac{1}{2}\sum_{x,y\in\mathbb{Z}^d}\sum_{i\in\{0,1\}}\rho\varepsilon^2(\eta(x)-\eta(y))^2\phi_\varepsilon(x,i)^2
\\&=\frac{\varepsilon^2p(1-p)}{1+(2\varepsilon+\varepsilon^2)p}\sum_{x\in\mathbb{Z}^d}\sum_{i\in\{0,1\}}\rho\phi_\varepsilon(x,i)^2
\\&\leq \frac{\varepsilon^2p(1-p)\rho}{1+(2\varepsilon+\varepsilon^2)p}=\frac{o(\varepsilon^2)}{1+o(\varepsilon)+o(\varepsilon^2)},\quad\varepsilon\to 0,
\end{align*}
and
\begin{align*}
III:&=\int_{\{0,1\}^{\mathbb{Z}^d}}\nu(\d\eta)\sum_{z\in\mathbb{Z}^d}\frac{1}{2}\sum_{y\in\mathbb{Z}^d,y\sim z}\kappa(f_\varepsilon(\eta,y,1)-f_\varepsilon(\eta,z,1))^2
\\&=\frac{\kappa}{2(1+(2\varepsilon+\varepsilon^2)p)}\int_{\{0,1\}^{\mathbb{Z}^d}}\nu(\d\eta)\sum_{z\in\mathbb{Z}^d}\sum_{y\in\mathbb{Z}^d,y\sim z}((1+\varepsilon\eta(y))\phi_\varepsilon(y,1)-(1+\varepsilon\eta(x))\phi_\varepsilon(z,1))^2
\\&=\frac{\kappa}{2(1+(2\varepsilon+\varepsilon^2)p)}\sum_{z\in\mathbb{Z}^d}\sum_{y\in\mathbb{Z}^d,y\sim z}(p^2(1+\varepsilon)^2+(1-p)^2)(\phi_\varepsilon(y,1)-\phi_\varepsilon(z,1))^2\\&+p(1-p)((1+\varepsilon)(\phi_\varepsilon(y,1)-\phi_\varepsilon(z,1))^2+(\phi_\varepsilon(y,1)-(1+\varepsilon)\phi_\varepsilon(z,1))^2
\\&=\kappa\sum_{z\in\mathbb{Z}^d}\sum_{y\in\mathbb{Z}^d,y\sim z}\left(\frac{1}{2}(1+(2\varepsilon+\varepsilon^2)p)(\phi_\varepsilon(y,1)-\phi_\varepsilon(z,1))^2+p(1-p)\varepsilon^2\phi_\varepsilon(y,1)\phi_\varepsilon(z,1)\right)
\\&\leq \frac{1}{2}\kappa(1+(2\varepsilon+\varepsilon^2)p)\varepsilon^2+2d\kappa\varepsilon^2p(1-p)= o(\varepsilon^2),
\end{align*}
where we used the relation
\begin{align*}
\phi_\varepsilon(y,1)\phi_\varepsilon(z,1)\leq\frac{1}{2}\phi_\varepsilon(y,1)^2+\frac{1}{2}\phi_\varepsilon(z,1)^2.
\end{align*}
Finally,
\begin{align*}
IV:&=\int_{\{0,1\}^{\mathbb{Z}^d}}\nu(\d\eta)\sum_{z\in\mathbb{Z}^d}2\sqrt{s_0s_1}(f_\varepsilon(\eta,z,1)-f_\varepsilon(\eta,z,0))^2
\\&=\frac{2\sqrt{s_0s_1}}{1+(2\varepsilon+\varepsilon^2)p}\sum_{z\in\mathbb{Z}^d}(p(1+\varepsilon)^2+(1-p))(\phi_\varepsilon(z,1)-\phi_\varepsilon(z,0))^2
\\&=2\sqrt{s_0s_1}\sum_{z\in\mathbb{Z}^d}(\phi_\varepsilon(z,1)-\phi_\varepsilon(z,0))^2\leq 2\sqrt{s_0s_1}\varepsilon^2=o(\varepsilon^2),\quad \varepsilon\to 0.
\end{align*}
Altogether we obtain
\begin{align}\label{transient}
\lim_{t\to\infty}\frac{1}{t}\log\left<U(t)\right>&\geq\sqrt{s_0s_1}+a_1\left(p\frac{\gamma +2\varepsilon+\varepsilon^2}{1+(2\varepsilon+\varepsilon^2)p}-s_1+s_0\right)-s_0+o(\varepsilon)\nonumber
\\&>\sqrt{s_0s_1}+a_1(\gamma p -s_1+s_0)-s_0, \quad \varepsilon\to 0,
\end{align}
since $p\in(0,1)$. If $s_0=s_1$, then the right hand-side of \eqref{transient} reduces to $a_1\gamma$, which can be maximized by choosing $a_1=1$. In case where $s_1>s_0$, the right hand-side of \eqref{transient} is strictly larger than $a_1(\gamma p-s_1+s_0)$. If $\gamma p-s_1+s_0\geq 0$, the optimal choice is $a_1=1$, and $a_0=0$ otherwise. For the case where $s_1<s_0$, the right hand-side of \eqref{transient} is strictly larger than $s_1-s_0+a_1(\gamma p -s_1+s_0)$. Since $\gamma p-s_1+s_0>0$, the optimal choice is $a_1=1$ in this condition. Thus, the following bounds can be established:
\begin{align*}
\lim_{t\to\infty}\frac{1}{t}\log\left<U(t)\right>>\left\{\begin{array}{ll}
\gamma p, &s_1\leq s_0,\\\gamma p -s_1+s_0, &s_1>s_0 \text{ and } \gamma p-s_1+s_0\geq 0,\\0, &s_1>s_0 \text{ and } \gamma p-s_1+s_0<0.\end{array}\right.
\end{align*}
For the upper bound, we use the representation \eqref{zeta} of the exclusion process as an environment seen from the walker. Let 
\begin{align*}
\Phi(\eta, i):=-s_0\delta_0(i)-s_1\delta_1(i)+i\gamma\eta(0) 
\end{align*}
and $\tau_y$ the usual shift operator as defined in \eqref{shift}. Recall the operator $V$ defined in \eqref{V} for which
\begin{align*}
V(\eta, x, i)=\Phi(\tau_x\eta, i)
\end{align*}
hold. Therefore, in a manner analogous to the proof of Theorem 1.3, 
\begin{align*}
\mathbb{E}_{\nu,(0,1)}^{(\tilde{\xi},\tilde{X},\tilde{\alpha})}\left[\exp\left(\int_0^tV(\xi(s),\tilde{X}(s),\tilde{\alpha}(s))\d s\right)\right]=\mathbb{E}_{\nu,1}^{\tilde{\zeta},\tilde{\alpha}}\left[\exp\left(\int_0^t\Phi((\tau_{X_s}\tilde{\xi})(\cdot,s), \tilde{\alpha}(s))\d s\right)\right]\leq\tilde{\lambda},
\end{align*}
where $\tilde{\lambda}$ is the largest eigenvalue of the self-adjoint operator $\Phi+\tilde{L}_{\text{EP}}$. The Rayleigh-Ritz formula yields
\begin{align*}
\tilde{\lambda}&=\sup_{\substack{f\in\ell^2\left(\{0,1\}^{\mathbb{Z}^d}\times\{0,1\}\right)\\\|f\|_2=1}}\left(\left<f,(\Phi+\tilde{Q}+\tilde{L}_{\text{SE}})f\right>+\int_{\{0,1\}^{\mathbb{Z}^d}}\nu(\d\eta)\sum_{y\sim 0}\kappa(f(\tau_y\eta,1)-f(\eta,1))f(\eta, 1)\right)
\\&\leq \sup_{\substack{f\in\ell^2\left(\{0,1\}^{\mathbb{Z}^d}\times\{0,1\}\right)\\\|f\|_2=1}}\left(\int_{\{0,1\}^{\mathbb{Z}^d}}\nu(\d\eta)\sum_{i\in\{0,1\}}\Phi(\eta,i)f(\eta,i)^2+\left<f,(\tilde{L}_{\text{SE}}+\tilde{Q})f\right>\right)
\\&=\mathbb{E}_1^{\tilde{\alpha}}\tilde{\mathbb{E}}_\nu\left[\exp\left(\int_0^t\Phi(\xi)(\cdot,s), \alpha(s))\d s\right)\right],
\end{align*}
where we used the fact that 
\begin{align*}
\int_{\{0,1\}^{\mathbb{Z}^d}}\nu(\d\eta)\sum_{y\sim 0}\kappa(f(\tau_y\eta,1)-f(\eta,1))f(\eta, 1)=-\int_{\{0,1\}^{\mathbb{Z}^d}}\nu(\d\eta)\sum_{j=1}^d\kappa(f(\tau_{e_j}\eta,1)-f(\eta,1))^2\leq 0
\end{align*}
in dimensions $d\geq 3$ (cf.\,\cite{upper}). At this point, the random walker $\tilde{X}$ does not appear in the formulas any more, and by applying a change of measure for $\alpha$ once again, 
\begin{align*}
\mathbb{E}_{\nu,1}^{\tilde{\zeta},\tilde{\alpha}}\left[\exp\left(\int_0^t\Phi(\xi)(\cdot,s), \alpha(s))\d s\right)\right]&=\mathbb{E}_{\nu,1}^{\tilde{\zeta},\tilde{\alpha}}\left[\exp\left(-s_0\tilde{L}_t(0)-s_1\tilde{L}_t(1)+\gamma\int_0^t\tilde{\alpha}(s)\tilde{\xi}(0,s)\d s\right)\right]
\\&=\e^{-\sqrt{s_0s_1}}\mathbb{E}_{\nu,1}^{(\xi,\alpha)}\left[\exp\left(\gamma\int_0^t\alpha(s)\xi(0,s)\d s\right)\right]
\\&\leq \e^{-\sqrt{s_0s_1}}\mathbb{E}_\nu^{\xi}\left[\exp\left(\gamma\int_0^t\xi(0,s)\d s\right)\right].
\end{align*}
Altogether and using Varadhan's lemma, we can deduce
\begin{align*}
\lim_{t\to\infty}\frac{1}{t}\log\mathbb{E}_{\nu,(0,1)}^{(\xi,X,\alpha)}\left[\exp\left(\gamma\int_0^t\alpha(s)\xi(X(s),s)\d s\right)\right]&\leq \lim_{t\to\infty}\frac{1}{t}\log\mathbb{E}_{\nu}^{\xi}\left[\exp\left(\gamma\int_0^t\xi(0,s)\d s\right)\right]
\\&=\sup_{x\in[0,1]}\{\gamma x -I_\xi(x)\}
\\&=\gamma-\inf_{x\in[0,1]}\{I_\xi(x)+\gamma(1-x)\}.
\end{align*}
Recalling $I_\xi$ defined in \eqref{Ixi}, we observe that $I_\xi(x)+\gamma(1-x)>0$ for all $x\in[0,1]$, and thus the upper bound follows. 
\subsection*{Acknowledgement}
The author would like to thank Professor Wolfgang König for his invaluable support.

\subsection*{Competing Interests}
The author has no competing interests to declare that are relevant to the content of this article.

\end{document}